\documentclass[openacc]{rsproca_new}
\usepackage{xcolor}
\usepackage{orcidlink}
\def\dvg{{\rm div}}
\def\xdif{{\rm d}}
\def\bs{\boldsymbol}
\newcommand{\R}{\mathbb R}
\newcommand{\ds}{\displaystyle}

\def\p{\bs p}

\def\M{\bs m^{\epsilon}}

\def\K{\bs K}

\def\A{\bs A}
\def\h{\bs h^{\epsilon}}

\def\mpsi{\bs \psi}
\def\dep{\partial}
\def\m{\bs m}
\def\e{\bs e}

\def\curl{{\rm curl\ }}

\newtheorem{theorem}{\bf Theorem}[section]
\newtheorem{proposition}{\bf Proposition}[section]
\newtheorem{corollary}{\bf Corollary}[section]
\newtheorem{remark}{\bf Remark}[section]
\newtheorem{definition}{\bf Definition}[section]
\newtheorem{lemma}{\bf Lemma}[section]

\titlehead{Research}

\begin{document}

\title{Homogenized phase transition model for perforated ferromagnetic media}
\author{
Catherine Choquet$^{1}$, Mohamed Ouhadan$^{1,2}$ and Mouhcine Tilioua$^{2}$}

\address{$^{1}$MIA Laboratory, University of La Rochelle, La Rochelle, France\\
$^{2}$MAMCS Team, FST Errachidia, University of Meknes, Errachidia, Morocco\\
\orcidlink{} CC, \href{https://orcid.org/0000-0003-0632-0358}{0000-0003-0632-0358}, MO, \href{https://orcid.org/0000-0002-9117-0448}{0000-0002-9117-0448}, MT, \href{https://orcid.org/0000-0003-1983-9669}{0000-0003-1983-9669} }

\subject{mathematical
	modelling, applied mathematics}


\corres{Catherine Choquet\\
\email{cchoquet@univ-lr.fr}}

\begin{abstract}
	This work presents a rigorous prediction of the effective equations governing the paramagnetic-ferromagnetic phase transition in a perforated three-dimensional body. Assuming a periodic distribution of perforations, we investigate the asymptotic behavior of solutions to the equations describing the thermodynamic and electromagnetic properties of the material as the period of the microstructure tends to zero. The microscopic model is a phase-field model within the Ginzburg-Landau framework for second-order phase transitions, where the phase-field is directly related to the magnetization vector. This model couples a nonlinear equation for the magnetization with the quasi-static Maxwell system and another nonlinear equation for the temperature. The primary mathematical challenge lies in homogenizing these equations, which exhibit a complex doubly non-linear structure. Additionally, the extension operators used within the homogenization framework precludes the application of standard Aubin--Lions compactness arguments. Our analysis employs two-scale convergence in conjunction with a two-scale decomposition based on an appropriate dilation operator. The nonlinearities are primarily addressed by means of a variant of compensated compactness and a Vitali compactness argument. From the perspective of practical applications, this work enables the explicit calculation of a Curie temperature tensor, capturing at the macroscopic scale the coupled effect of the material's geometric structure and its magnetic permeability tensor.
\end{abstract}

\maketitle

\section{Introduction}	

The theory of periodic homogenization focuses on partial differential equations whose coefficients exhibit periodic oscillations with a small period, denoted by $\epsilon$. The main objective of this theory is to derive a homogenized equation that describes the global behavior of the system as the oscillation period tends to zero. In other words, it seeks to establish an equivalent partial differential equation, whose solutions represent the weak limits of the original equation's solutions as $\epsilon$ decreases. For a comprehensive introduction to this theory, as well as general surveys, the works \cite{bensoussan-lions-papanicolaou,cioranesco-donato,jik,mar,tar} provide essential references. 
A major milestone in this theory was the introduction of the concept of two-scale convergence by \cite{nguestseng}, which paved the way for the analysis of more complex systems. In its initial version, this convergence, explored in \cite{nguestseng,allaire}, is referred to as ``weak two-scale convergence". However, it became crucial to define a more robust form of convergence, known as ``strong two-scale convergence", as presented in the works \cite{vin,mie}. This strong convergence is particularly valuable for dealing with fully nonlinear problems, such as nonsmooth elastoplasticity, discussed in \cite{nes,vin1,sch,gia,han}, and also allows for more efficient numerical approximation of such problems, as mentioned in \cite{mat}.

The present work applies the concepts of  two-scale convergence to a phenomenon of  nonlinear phase transitions, which play a significant role in condensed matter physics. This area is rich with examples of such transitions, including magnetic, ferroelectric, superfluid, and superconducting transitions. We are interested in phase transitions in ferromagnetism, which refers to the ability of certain materials to become magnetized in the presence of an external magnetic field. In this state, the magnetic moments align in the same direction as the field, and even after the field is removed, part of this alignment is retained. However, when the temperature exceeds a specific threshold, known as the Curie temperature $\theta_{c}$, the residual alignment disappears, and the material reverts to the paramagnetic phase. The transition from the paramagnetic to the ferromagnetic state is modelled as a second-order phase transition (see \cite{brokate} and \cite{landau}).
For the mathematical model, our starting point is the work of Berti et al. in \cite{berti} where a system of partial differential equations describing the coupled behavior of magnetization, temperature and magnetic field is provided.

We are interested in ferromagnetic materials with periodic perforated structure. 
The characteristic size of the holes is assumed to be small, denoted by $\epsilon$. As emphasized in the seminal paper \cite{cioranescu-saint}, the homogenization process in this context involves letting $\epsilon$ tend to zero. This presents two key challenges: (1) handling the rapid oscillations within the material due to the presence of the holes, and (2) accounting for the changes in the domain structure as the number of holes increases.	In this work we shall make use of the two-scale convergence method as described in \cite{allaire,pavliotis-stuart}. We also adopt the formalism introduced by Visintin in \cite{vin} which may be viewed as the use of an appropriate dilation operator as in \cite{arbo,choquet1}.

The paper is organized as follows: 
The perforated geometry  and  the model for paramagnetic-ferromagnetic  phase transitions are described in Section \ref{setting}. 
We also introduce various notations. 
\textcolor{blue}{Since we essentially use properties of two-scale convergence, the basics of this concept are recalled in Section \ref{sec!}, together with other auxiliary tools.  }
The main result is stated in Section  \ref{sec3}.
It is the effective model describing the phase transition, justified by the analysis of the asymptotic behavior of the solutions of the micromodel in a periodically perforated domain. 
The rest of the paper is devoted to the proof of the main result. 
In Section \ref{sec4}, the necessary uniform estimates are established. 
Section \ref{sec5} is finally dedicated to the homogenization process, {\it i.e.}, the passage to the limit in the different equations.

\section{Setting of the problem and preliminary results}  
\label{setting}
\subsection{Notations}\label{notata}

	Let $\Omega$ be a bounded domain in $\R^{3}$  with a smooth boundary $\partial\Omega$ and $Y = (0,1)^{3}$ the reference cell  of
periodicity in the auxiliary space  $\R^{3}$. 
The canonical basis of $\R^{3}$ is $(e_{1}, e_{2}, e_{3})$.
Throughout the paper, the small parameter $\epsilon$  takes its values in a sequence of positive real numbers tending to zero.
Let $\mathcal{H}$, the reference hole, be an open subset of $Y$ with a smooth boundary
$\partial \mathcal{H}$ and set $Y^{\ast}=Y\setminus \bar{\mathcal{H}}$. 
For a simple  illustration see Figure \ref{figure1}.	
\textcolor{blue}{However our results may apply to the general case of periodic holes which may not be isolated, as in \cite{allaire1}.}
We set
\begin{equation*}\label{Do}
	\bar{\mathcal{H}}^{\epsilon}=\bigcup_{k\in\mathbb{Z},\epsilon(k+Y)\subset\Omega}\epsilon(k+\bar{\mathcal{H}}), 
	\qquad
	\Omega^{\epsilon}=\Omega\setminus \bar{\mathcal{H}}^{\epsilon} .
\end{equation*}
\textcolor{blue}{Also, for shortening notations when using integration, we introduce the sets} $\Omega^{\epsilon}_{T}=(0,T)\times \Omega^{\epsilon}$ and  $\Omega_{T}=(0,T)\times \Omega$ for $T > 0$. 
Let $\nu$ (resp. $\nu^\epsilon$) be the outer unit normal at the boundary $\dep\Omega$ (resp. $\dep\Omega^{\epsilon}$).
\begin{figure}[pb]
	{\centering
		\includegraphics[width=3in]{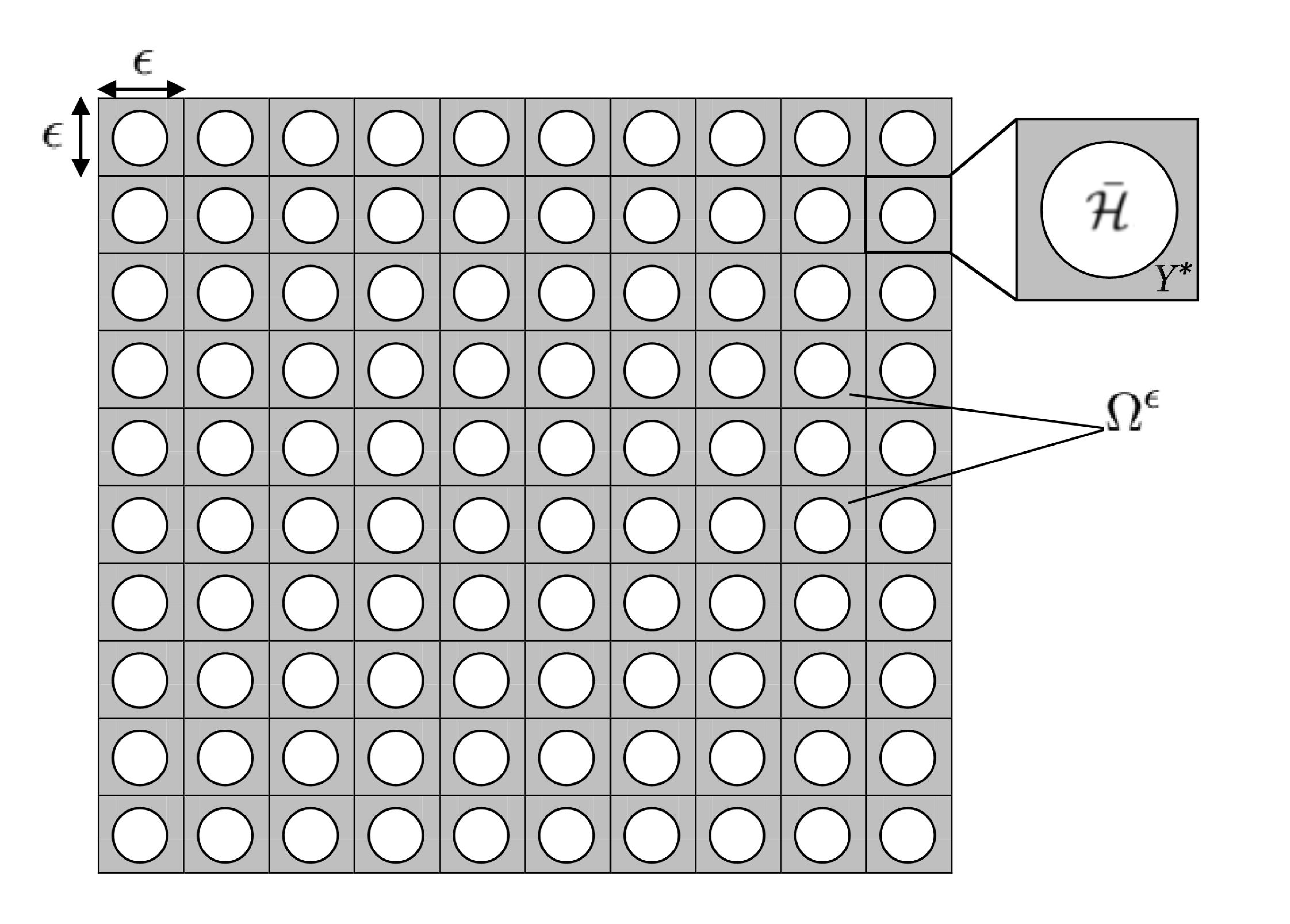}
		\caption{Example of perforated domains $\Omega^{\epsilon}$ and $Y^{\ast}$ in 2-D.}}
	\label{figure1}
\end{figure}

 \textcolor{blue}{Throughout the paper, we denote by $\chi_O$ the characteristic function of a set $O \subset \mathbb{R}^3$. We define the mean value $\bar{\chi}$  of the function $\chi_{Y^\ast}$} on $Y$:
$$\bar{\chi}=\mathcal{M}_{Y}(\chi_{Y^\ast}) = \int_Y \chi_{Y^\ast}(y)\, \xdif y .$$

We now introduce some functional spaces and other notations for the mathematical framework of our problem.
We denote by $C^{\infty}_{\sharp}(Y)$ the set of infinitely differentiable real functions that are $Y$-periodic in each of the three space variables. 
Space $H^{1}_{\sharp}(Y)$ is the closure of $C^{\infty}_\sharp(Y)$ in $H^{1}(Y)$.
Also  $C^{\infty}(\overline{\Omega})\otimes C^{\infty}_{\sharp}(Y)$  denotes the set containing all infinitely differentiable real functions over $\overline{\Omega}\times\R^{3}$ that are $Y$-periodic in the three last variables. 

For any normed vector space $Z$,  $\|\cdot\|_{Z}$ denotes the $Z$-norm. 
To simplify notations, if all the entries of a matrix-valued function $\mathbf{f}$ : $\R^d \to \R^{p\times q}$ belong to a functional space $Z$, we write $\mathbf{f} \in Z$ instead of  $\mathbf{f} \in Z^{p\times q}$. 
Capital $C$  represents various nonnegative real numbers independ of $\epsilon$.

\medskip

Let $\A(x,y)=(A_{ij}(x,y)),\K(x,y)=(K_{ij}(x,y))$, $\bs{\mu}(x,y)=(\mu_{ij}(x,y))$, $1\leq i,j\leq3$, \textcolor{blue}{be three symmetric positive definite matrix fields with values in $\mathbb{R}^{3 \times 3}$. We assume that $\A$ and $\K$ belong to the class $L^{\infty}(\Omega; C_{\sharp}(Y))$, and that $\bs{\mu} \in L^{\infty}(\mathbb{R}^3; C_{\sharp}(Y))$.
} 
We assume also that $\A$, $\K$ and $\bs{\mu}$ are coercive: there exist $\alpha_{1},\alpha_{2},\alpha_3>0$ such that 
\begin{equation*}\label{hy1}
	\A \xi \cdot \xi \geq \alpha_{1}  \vert \xi \vert^2, 
	\quad
	\K \xi \cdot \xi \geq \alpha_{2}  \vert \xi \vert^2, 
	\quad
	\bs{\mu}  \xi \cdot \xi \geq \alpha_{3}  \vert \xi \vert^2,
\end{equation*}
for all $\xi$ in $\R^{3}$, for all $y \in Y$, for almost any $x$ in $\Omega$  or in $\mathbb{R}^3 $.
We set $\A^{\epsilon}(x)=\A\big(x,\frac{x}{\epsilon}\big)$, $\K^{\epsilon}(x)=\K\big(x,\frac{x}{\epsilon}\big)$, $\bs{\mu}^{\epsilon}(x)=\bs{\mu}\big(x,\frac{x}{\epsilon}\big)$.

\subsection{Microscopic model}

Let us now describe the model equations. We assume that a ferromagnetic material occupies the domain $\Omega^{\epsilon}$.
At the microscopic scale, we adapt the model derived in  the paper \cite{berti} of Berti {\it et al.}  (see also Belov \cite{belov}).  
It combines phenomenological constitutive equations for the magnetization $\M$, the magnetic field $\h$ and the absolute temperature $\theta^{\epsilon}$. 
We enrich somehow the model proposed in \cite{berti} by introducing some anisotropy in the problem  (about anisotropy, see also Remark \ref{rem_anis} below).
More precisely we consider two heterogeneity scales in the ferromagnetic material. 
First, using $\A$, $\K$ and $\bs{\mu}$ defined before,  we replace scalar coefficients by matrices. 
Then, most importantly, microscopic heterogeneity is modelled by oscillations, here assumed of period of order $\epsilon$, thus the introduction of $\ds\A^{\epsilon}$,  $\K^\epsilon$ and $\bs{\mu}^{\epsilon}$.

On the one hand, the temperature $\theta^{\epsilon}$ is ruled by the following equation
\begin{equation}\label{new_depart}
	\partial_t (c(\theta^{\epsilon})) - \theta^{\epsilon}  \M \cdot \partial_t \M - \mathrm{div} (\K^{\epsilon}(x)k(\theta^{\epsilon}) \nabla \theta^{\epsilon}) =  r,
\end{equation}
where the functions  $c(\theta^{\epsilon})$ and $k(\theta^{\epsilon})$ are respectively the thermal conductivity and the specific temperature of the material and $\K^{\epsilon}$ is a thermal diffusion matrix that depends on the characteristics and geometry of the material. 
\textcolor{blue}{Studies on phase transitions consider various laws for heat conductivity and specific heat (see, for example, \cite{berti2, brokate}). Here, we focus on the case studied by Berti {\it et al.}  \cite{berti}, where heat conductivity and specific heat depend on the absolute temperature according to polynomial laws}, as follows:
\begin{equation}\label{heatlaw}
	c(\theta^{\epsilon}) = c_1\theta^{\epsilon} + c_2(\theta^{\epsilon})^2/2, \qquad k(\theta^{\epsilon})=k_0 + k_1 \theta^{\epsilon},
\end{equation}
where $c_1,c_2,k_0,k_1>0$.
Like \cite{berti} also, we develop the derivatives in \eqref{new_depart}  and neglect the quadratic terms in $\partial_i \theta^{\epsilon}$.
Dividing by $\theta^\epsilon$, we get 
\begin{equation}\label{depart1}
	c_{1}\dep_t\ln(\theta^{\epsilon})+c_{2}\dep_t\theta^{\epsilon} -   \M \cdot \partial_t \M - \mathrm{div} (\K^{\epsilon}\nabla(k_0\ln(\theta^{\epsilon}) +k_1 \theta^{\epsilon})) = \hat r
\end{equation}
with $\ds\hat r=r/\theta^{\epsilon}$. 
For simplicity,  assume  $\hat{r}\equiv0$.

On the other hand, the evolution of the ferromagnetic material is characterized by the magnetization $\M$ which depends on the temperature $\theta^\epsilon$ as follows: 
\begin{equation}\label{model1}
	\gamma\dep_t \M=\dvg(\A^{\epsilon}\nabla\M)-\theta_{c}(|\M|^{2}-1)\M-\theta^{\epsilon}\M+\h~~~~\mbox{in}~\Omega^{\epsilon}_{T} .
\end{equation}
 \textcolor{blue}{The gyromagnetic ratio} 
$\gamma$ is a given positive real number and $\theta_{c}$ is the Curie temperature. 
The magnetic field $\h$ is the stray field that appears in the Maxwell's equations. The magnetization $\M$ links the magnetic field $\h$ and the magnetic
induction $\mathbf{B}^{\epsilon}$ by the relation 
$\mathbf{B}^{\epsilon} = \bs{\mu}^{\epsilon}\h + \chi^{\epsilon}\M,$ 
where $\bs{\mu}^{\epsilon}$ represents the magnetic permeability. 
The magnetic field $\h$ satisfies $\curl\h = 0$ according to static Maxwell's equations, and, according to  Faraday law,  $\dvg\mathbf{B}^{\epsilon} = 0$. 
Hence the magnetization $\M$ induces a magnetic field $\h$ which is given by
\begin{equation*}\label{model3}
	\begin{array}{ll}
		\curl(\h) = 0 , \quad \nabla \cdot \left( \boldsymbol{\mu}^\epsilon \h + \chi_{\Omega^\epsilon} \M \right) = 0  & \text{in } \mathbb{R}_+ \times \mathbb{R}^3,\\
		\left( \boldsymbol{\mu}^\epsilon \h + \chi_{\Omega^\epsilon} \M \right) \cdot \nu^\epsilon = 0 & \text{on } (0,T) \times \partial\Omega^\epsilon.
	\end{array}
\end{equation*}
The magnetic field $\h$ thus derives from a scalar magnetic potential $\varphi^{\epsilon}$,  $\h=\nabla\varphi^{\epsilon}$, which satisfies
\begin{equation}\label{model4}
	\left\{\begin{array}{ll} 
	\dvg(\bs{\mu}^{\epsilon}\nabla\varphi^{\epsilon} + \chi_{\Omega^\epsilon}\M)=0 ~~~~\mbox{in} ~\R_+ \times \R^{3},\\
		{\color{blue}(\bs{\mu}^{\epsilon}\nabla\varphi^{\epsilon} + \chi_{\Omega^\epsilon}\M)\cdot\nu^\epsilon=0~~\mbox{ on } (0,T)\times\dep\Omega^{\epsilon}}.
	\end{array}\right.
\end{equation}
%
%
%
%
%
%
 Combining equations \eqref{depart1}-\eqref{model4} completed by appropriate initial and boundary conditions,  we get the following problem ruling the evolution of the ferromagnetic material:
\begin{subequations}\label{modelfit}
	\begin{flalign}
		&\gamma\dep_t \M=\dvg(\A^{\epsilon}\nabla\M)-\theta_{c}(|\M|^{2}-1)\M-\theta^{\epsilon}\M+\nabla\varphi^{\epsilon}\hspace{1.3cm}\mbox{ in }\Omega^{\epsilon}_{T}, &
		\label{1.a} 
		\\
		&c_{1}\dep_{t}\ln(\theta^{\epsilon})+c_{2}\dep_t\theta^{\epsilon}-\M\cdot\dep_t\M=\dvg(\K^{\epsilon}\nabla(k_0 \ln(\theta^{\epsilon}) +k_1 \theta^{\epsilon}))\hspace{0.5cm}\mbox{ in } \Omega^{\epsilon}_{T},&
		\label{1.b} \\
		&\dvg(\bs{\mu}^\epsilon\nabla\varphi^{\epsilon} + \chi_{\Omega^\epsilon}\M)=0 \hspace{4.8cm}\mbox{ in }(0,T)\times\R^{3},&
		\label{1.c} 
		\\
		&	(\A^{\epsilon}\nabla\M)\cdot\nu=0   , \quad
		(\K^{\epsilon}\nabla(k_0 \ln(\theta^{\epsilon}) +k_1 \theta^{\epsilon}))\cdot \nu^\epsilon = 0   \hspace{0.6cm} \mbox{ on } (0,T)\times \dep\Omega^{\epsilon} ,&
		\label{1.d} 
		\\
		&(\bs{\mu}^\epsilon\nabla\varphi^{\epsilon}+ \chi_{\Omega^\epsilon}\M)\cdot\nu^\epsilon=0 \hspace{4.8cm} \mbox{ on } (0,T)\times\dep\Omega^{\epsilon} ,&
		\label{1.e} 
		\\
		&	\M(x,0)=\M_0(x),\ \theta^{\epsilon}(x,0)=\theta^{\epsilon}_{0}(x)  \hspace{3.2cm}\mbox{in} ~\Omega^{\epsilon}.&
		\label{1.f} 
	\end{flalign}
\end{subequations}
Additionally,  the initial magnetic field $\nabla\varphi^\epsilon_0$ is also ruled by the quasi-static Maxwell system:
\begin{equation}\label{inmag}
	\left\{\begin{array}{cc}
		\ds\dvg(\bs{\mu}^\epsilon\nabla\varphi^{\epsilon}_0 + \chi_{\Omega^\epsilon}\M_0)=0 &\mbox{ in }\R_+ \times \R^{3},\\
		\ds(\bs{\mu}^\epsilon\nabla\varphi^{\epsilon}_0+ \chi_{\Omega^\epsilon}\M_0)\cdot\nu^\epsilon=0 &  \mbox{ on } \R_+ \times \dep\Omega^{\epsilon}.
	\end{array}\right.
\end{equation}

The mathematical analysis developed by Berti {\it et al.} in \cite{berti} (see also Tilioua in \cite{tilioua}) extents straightforward to the microscopic problem \eqref{modelfit}.
Introduce the functional $\mathcal{E}^{\epsilon}(t)$ defined as
\begin{equation*}\label{feng1}
	\begin{array}{ll}
		\mathcal{E}^{\epsilon}(t)=\ds\frac{1}{2}\Big[\int_{\Omega^{\epsilon}} \A^{\epsilon}\nabla
		\M\cdot\nabla\M \ \xdif x+\frac{\theta_{c}}{2}\|\M\|^{4}_{L^{4}(\Omega^{\epsilon})}+c_{2}\|\theta^{\epsilon}\|^{2}_{L^{2}(\Omega^{\epsilon})}\\[2ex]
		~~~~~~~~~~~~~~~~~~\ds+\int_{\R^3} \bs{\mu}^{\epsilon}\nabla
		\varphi^{\epsilon}\cdot\nabla\varphi^{\epsilon}\ \xdif x+2c_{1}\ds\int_{\Omega^{\epsilon}}\theta^{\epsilon} \ \xdif x\Big], \quad t>0,
	\end{array}
\end{equation*}
and, for time  $t = 0$,
\begin{equation*}\label{feng0}
	\begin{array}{ll}
		\mathcal{E}^{\epsilon}(0)=\ds\frac{1}{2}\Big[\int_{\Omega^{\epsilon}} \A^{\epsilon}\nabla
		\M_{0}\cdot\nabla\M_{0}\ \xdif x+\frac{\theta_{c}}{2}\|\M_{0}\|^{4}_{L^{4}(\Omega^{\epsilon})}+c_{2}\|\theta^{\epsilon}_{0}\|^{2}_{L^{2}(\Omega^{\epsilon})}\\[2ex]
		~~~~~~~~~~~~~~~~~~\ds+\int_{\R^3} \bs{\mu}^{\epsilon}\nabla
		\varphi_0^{\epsilon}\cdot\nabla\varphi_0^{\epsilon}\ \xdif x +2c_{1}\ds\int_{\Omega^{\epsilon}}\theta^{\epsilon}_{0}\ \xdif x\Big] .
	\end{array}
\end{equation*}
The following existence result holds true.
\begin{theorem}[Global existence]{\normalfont (Berti {\it et al.} \cite{berti}, Tilioua \cite{tilioua})}\label{exist}
	Let $\M_0\in H^1(\Omega^{\epsilon})$ and $\theta_0^\epsilon \in L^2(\Omega^\epsilon)$.
	Under the aforementioned hypothesis, for every $T > 0$,
	Problem \eqref{modelfit} admits a weak solution $(\M,\theta^{\epsilon},\varphi^{\epsilon})$ such that 
	$$
	\begin{array}{ll}
		\M\in L^{2}(0,T;H^{2}(\Omega^{\epsilon})) \textcolor{blue}{\cap L^\infty(0,T;L^4(\Omega^\epsilon))}  \cap H^{1}(0,T;L^{2}(\Omega^{\epsilon})),\\[2ex]
		\theta^{\epsilon}\in L^{2}(0,T;H^{1}(\Omega^{\epsilon})),~\ln(\theta^{\epsilon})\in L^{2}(0,T;H^{1}(\Omega^{\epsilon})),\\[2ex]
		c_{1}\ln(\theta^{\epsilon})+c_{2}\theta^{\epsilon} \in H^{1}(0,T;(H^{1}(\Omega^{\epsilon}))^{'}),\\[2ex]
		\nabla\varphi^{\epsilon}\in L^{\infty}(0,T;L^{2}(\R^{3}\setminus \bar{\mathcal{H}}^{\epsilon})).
	\end{array}
	$$
	Moreover the following energy estimate holds for all $t \in (0,T)$:
	\begin{equation}\label{energy-estimate}
		\mathcal{E}^{\color{blue}\epsilon}(t)+\ds\int_0^t\Big(\frac{\gamma}{2}\|\dep_t\M\|_{L^{2}(\Omega^{\epsilon})}^2+\ds\int_{\Omega^{\epsilon}}\K^{\epsilon}\nabla(k_0 \ln(\theta^{\epsilon}) +k_1 \theta^{\epsilon})\cdot\nabla\theta^{\epsilon} \ \xdif x\Big)\ \xdif s
		\leq \eta_{1}{\cal E}^{\color{blue}\epsilon}(0)+\eta_{2}
	\end{equation}
	where $\eta_{1}$ and $\eta_{2}$ are suitable nonnegative real numbers depending on $T$.
\end{theorem}

{\color{blue}\begin{remark}
The assumption that the initial magnetic field satisfies \eqref{inmag} is introduced to ensure the physical compatibility of the initial data with the magnetostatic constraint, and to formulate the coupled evolution problem from the outset. Mathematically, this assumption guarantees that the magnetic potential $\varphi^\epsilon_0$ is well-defined, so that the total energy $\mathcal{E}^\epsilon$ is finite at the initial time $t = 0$, provided that $\M_0 \in L^2(\Omega^\epsilon)$. 
Indeed, following Amrouche \textit{et al.}~\cite{amrouche}, we introduce the weighted Sobolev space $ W_{bp}(\mathbb{R}^{3} \setminus \bar{\mathcal{H}}^{\epsilon})$, adapted to our geometry, and referred to as the Beppo-Levi space, defined by
	\begin{gather*}
		W_{bp}(\R^{3}\setminus\bar{\mathcal{H}}^{\epsilon})
		=\Big\{\omega\in \mathcal{D}^{'}(\R^{3}\setminus\bar{\mathcal{H}}^{\epsilon})\mbox{ such that }\nabla\omega\in L^{2}(\R^{3}\setminus\bar{\mathcal{H}}^{\epsilon})
		\mbox{ and}~\frac{\omega(x)}{\sqrt{1+|x|^{2}}}\in L^{2}(\R^{3}\setminus\bar{\mathcal{H}}^{\epsilon})\Big\},
	\end{gather*}
	and equipped with scalar product
	$\ds(\omega,\psi)_{W_{bp}(\R^{3}\setminus\bar{\mathcal{H}}^{\epsilon})}=\int_{\R^{3}\setminus\bar{\mathcal{H}}^{\epsilon}}\nabla\omega\cdot\nabla\psi\ \xdif x$.
	Using Lax-Milgram theorem (see, for example, \cite{rep}), one proves that, for any $\M_0$  given in $L^{2}(\Omega^\epsilon)$, Problem \eqref{inmag} has a unique solution in $ W_{bp}(\mathbb{R}^{3} \setminus \bar{\mathcal{H}}^{\epsilon})$.
	\\
Moreover, in order for $\mathcal{E}^\epsilon$ to be finite at time $t = 0$, it is also necessary that $\nabla \M_0 \in L^2(\Omega^\epsilon)$ and $\M_0 \in L^4(\Omega^\epsilon)$. Since the Sobolev space $H^1(\Omega^\epsilon)$ embeds compactly into $L^4(\Omega^\epsilon)$, it is  natural to assume that $\M_0 \in H^1(\Omega^\epsilon)$.
\end{remark}}
\begin{remark}\label{rem_anis}
	Let us provide some details regarding how we account for anisotropy in the model.
	According to Berti {\it et al.}  \cite{berti}, \eqref{1.a} should be replaced by
	$$ \gamma\dep_t \M = \dvg(\A^{\epsilon}\nabla\M) - \theta_{c}f_1'(\M) - \theta^{\epsilon}f_2'(\M)+\nabla\varphi^{\epsilon} $$
	with $f_1$ and $f_2$ in the form
	$$ f_1(\m) = \frac{1}{4} (\mathbb{F} \m \otimes \m)\cdot \m \otimes \m - f_2(\m), \quad f_2(\m) = \frac{1}{2} \mathcal{F} \cdot (\m \otimes \m)$$
	where $\mathbb{F}$ and $\mathcal{F}$ are fourth-order and second-order positive definite tensors, respectively. 
	The isotropic assumption consists in choosing for  $\mathbb{F}$ and $\mathcal{F}$ the identity tensor, thus recovering Eq. \eqref{1.a}, and in replacing $\A$ by a scalar function.
	In the present work, we use $\A$ for taking into account the anisotropy but we keep the isotropic form of $f_1$ and $f_2$.
	This latter choice is for the sake of clarity in the computations. We emphasize that the characteristic nonlinearities are represented by this choice. 
	The only important assumption in the general setting with $f_1$ and $f_2$ would be that the existence result in Theorem \ref{exist} remains true.
\end{remark}

\subsection{Reformulation of the homogenization problem}

We now aim  to study the asymptotic behavior of Problem \eqref{modelfit}  as $\epsilon$ tends to 0 in order to rigorously derive the corresponding homogenized model.
Classically the difficulty lies  in the nonlinear terms. 
Here the double nonlinearity in \eqref{1.b} makes the passage to the limit significantly complex. It is worth noting that this equation can also be written in the following form:
\textcolor{blue}{
\begin{equation*}
	\ds \dep_{t}F_1(\theta^{\epsilon})-\M\cdot\dep_t\M=\dvg F_2\left(x,\frac{x}{\epsilon},\theta^{\epsilon},\nabla \theta^{\epsilon}\right),
\end{equation*}}
 \textcolor{blue}{with $F_1 : \mathbb{R}_+ \to \mathbb{R}$ and $F_2 : \mathbb{R}^3 \times \mathbb{R}^3 \times \mathbb{R}_+ \times \mathbb{R}^3 \to \mathbb{R}^3$ defined by	$$F_1(s) = c_1 \ln(s) + c_2 s, \quad \text{and} \quad
 	F_2(x, y, s, \bs{\xi}) = \mathbf{K}(x, y) \left( \frac{k_0}{s} + k_1 \right) \bs{\xi}.$$
}
In general, the homogenization of this type of equations is based on certain conditions imposed on $F_i, \ i = 1, 2$. We cite, for example, \cite{Jian} for the case of homogenization of parabolic equations in fixed domains. We  also mention \cite{Nan1} and \cite{Nan2} for the case of perforated domains with Dirichlet and Neumann boundary conditions, respectively. \textcolor{blue}{In particular,  common assumptions imposed on $F_1$ are: (i) Function $F_1$ is continuous and non-decreasing, with $F_1(0) = 0$; (ii) There exists a constant $\delta > 0$ such that, for every $r$ and $S$ satisfying $0 < r < S$, there exists a constant $C(r, S) > 0$ such that
	$
	|F_1(s_1) - F_1(s_2)| \geq C(r, S) |s_1 - s_2|^\delta
	$
	for all $s_1, s_2 \in [-S, S]$ with $r < |s_1| < R$.
Assumptions (i) and (ii) play a crucial role in characterizing the homogenized model because they allow for strong convergence of the solution. For instance (i) is essential for proving the strong convergence of  $F_1$ valued in the zero extension of the solution (the nature of this extension varies depending on the context: in \cite{Nan1} and \cite{Nan2}, it refers to the extension by zero inside the perforated regions, whereas in \cite{Jian}, it concerns the extension outside a set defined in the proof of Lemma 1.9 in \cite{alt}).
 In our case,  $F_1$ do not satisfy (i) nor (ii), of course due to  the logarithmic function. }
  \textcolor{blue}{In brief, our concern is the possibly very singular behaviour of the logarithmic function that characterizes the time dynamics, instead of the nonlinearity or the possible degenerescence of the time derivative.} Therefore, we adopt a different strategy to pass to the limit.

The first step is a reformulation of the problem. 
Let $F$ be the function defined in $\R_{+}^*$ by 
\begin{equation}\label{fonction}
	F(x) = c_1\ln(x) + c_2x.
\end{equation}
It is clear that $F$ is bijective, as it is strictly increasing (injective) and its image covers all of $\R$ (surjective). Let $G = F^{-1}$, which is defined on $\R$ and is valued in $\R_{+}^*$.
Set
\begin{equation}\label{tra10}
	\theta^{\epsilon}=G(v^{\epsilon}), \quad v^\epsilon=c_{1}\ln(\theta^{\epsilon})+c_{2}\theta^{\epsilon}=F(\theta^\epsilon).
\end{equation} 
Notice that
\begin{equation}\label{tra20}
	\nabla(k_0 \ln(\theta^{\epsilon}) +k_1 \theta^{\epsilon})=g(\theta^\epsilon)\nabla v^\epsilon, \ \mbox{with }\ \ds g(x)=\dfrac{k_0+k_{1}x}{c_1+c_{2}x},\ \forall x>0.
\end{equation}
Hence, from  \eqref{tra10}-\eqref{tra20}, Problem \eqref{modelfit} is rewritten as
\begin{subequations}\label{modelfit1}
	\begin{flalign}
		& \gamma\dep_t \M=\dvg(\A^{\epsilon}\nabla\M)-\theta_{c}(|\M|^{2}-1)\M-G(v^{\epsilon})\M+\nabla\varphi^{\epsilon}\hspace{0.5cm}\mbox{ in }\Omega^{\epsilon}_{T},&
		\label{2.a}
		\\
		&\ds	\dep_{t}v^{\epsilon}-\M\cdot\dep_t\M=\dvg(\K^{\epsilon}g(G(v^{\epsilon}))\nabla  v^{\epsilon})\hspace{3.55cm}\mbox{ in } \Omega^{\epsilon}_{T},&
		\label{2.b}
		\\
		&	\dvg(\bs{\mu^{\epsilon}}\nabla\varphi^{\epsilon} + \chi_{\Omega^\epsilon}\M)=0 \hspace{4.86cm}\mbox{ in }(0,T) \times\R^{3},&
		\label{2.c}
		\\
		& (\A^{\epsilon}\nabla\M)\cdot\nu^\epsilon=0   ,\quad
		(	\K^{\epsilon}g(G(v^{\epsilon})\nabla  v^{\epsilon})\cdot \nu^\epsilon = 0   \hspace{1.32cm}  \mbox{ on } (0,T)\times \dep\Omega^{\epsilon}, &
		\label{2.d}
		\\
		&	(\bs{\mu^{\epsilon}}\nabla\varphi^{\epsilon}+ \chi_{\Omega^\epsilon}\M)\cdot\nu^\epsilon=0 \hspace{4.5cm} \mbox{ on } (0,T)\times\dep\Omega^{\epsilon} ,&
		\label{2.e}
		\\
		&\M(x,0)=\M_0(x),\ v^{\epsilon}(x,0)=v^{\epsilon}_{0}(x) \hspace{4.6cm}\mbox{in} ~\Omega^{\epsilon}.&
		\label{2.f}
	\end{flalign}
\end{subequations}
\textcolor{blue}{Function $F$ defined in \eqref{fonction} and relations \eqref{tra10}-\eqref{tra20} are presented here as a computational trick. Note that the same idea can be expressed in terms of linear operators and bilinear forms, involving an inversion of the special operator at the time derivative (see {\it e.g.}, Sanchez-Palencia \cite{Sanchez} Chapter 6, Section 4).} 
%
\section{Auxiliary tools}\label{sec!}
\textcolor{blue}{
The aim of the paper is now to pass to the limit $\epsilon \to 0$ in Problem \eqref{modelfit1}.
This section compiles the principal convergence analysis tools that will be utilized subsequently.
We begin by recalling in Subsection \ref{sub2} the definition of two-scale convergence, along with some fundamental results related to this notion. For a more detailed presentation, we refer the reader to the works of Allaire \cite{allaire} and Pavliotis and Stuart \cite{pavliotis-stuart} (see in particular Subsection 2.5.2). 
The strength of this method lies in its natural ability to capture the microscopic scale during the limiting process
However, passing to the limit in the nonlinear terms of Problem \eqref{modelfit1} requires a slightly modified formalization of the standard two-scale convergence method. 
We make use of a dilation operator, as proposed in \cite{arbo}. In this framework, we also adapt the two-scale decomposition method introduced in \cite{vin} to handle time-dependent problems. This approach, commonly referred to in the literature as the unfolding method (see \cite{ciora}), requires an extension of the classical framework to simultaneously account for both spatial and temporal variations.
It is described in Subsection \ref{2svi}.
The unfolding method transforms the problem into a fixed domain. It is thus renowned for allowing the use of standard notions of weak and strong convergence. However, we will primarily use it to concisely describe the limiting processes in the nonlinear terms. The essential compactness argument will first be obtained directly within the framework of two-scale convergence, using a compensated compactness argument between different types of solution extensions. We rely on a result by Moussa, recalled in Subsection \ref{mouSs}.}

\subsection{ Two-scale convergence}\label{sub2}
\textcolor{blue}{In this subsection, we present the two-scale convergence results needed for the asymptotic analysis of our problem. }
\begin{definition}\label{def}
	A function $f\in L^{2}(\Omega\times Y)$ is admissible if
	\begin{enumerate}
		\item the sequence $f^{\epsilon}(x)=f(x,x/\epsilon)$ is uniformly bounded in $L^{2}(\Omega)$;
		\item 
		$\ds \lim_{\epsilon\rightarrow 0}\int_{\Omega^{\epsilon}}\Big|f\Big(x,\frac{x}{\epsilon}\Big)\Big|^{2}\xdif x=\int_{\Omega\times Y}|f(x,y)|^{2}\xdif y\xdif x .$
	\end{enumerate}
\end{definition}
\begin{remark}
	The functional spaces $L^{2}(\Omega;C_{\sharp} (Y))$, $C(\Omega;C_{\sharp} (Y))$ and $L^{2}_{\sharp}(\Omega;C (Y))$ are spaces of admissible functions,
	which identify with dense subspaces of $ L^{2}(\Omega\times Y)$.
	Notice that $\A$, $\K$ and $\bs{\mu}$ are admissible functions.
\end{remark}
The two-scale convergence is designed for capturing high frequency oscillations of the micro-scale by using resonance with the oscillations in admissible test functions.
\begin{definition} A sequence $u^{\epsilon}\in L^{2}(\Omega_{T})$ two-scale converges to $u^{0}\in L^{2}(\Omega_{T}\times Y)$, denoted in the sequel
	$$   u^\epsilon    \overset{2}{\rightharpoonup} u^0 , $$ 
	if for every test function $\phi\in L^{2}(\Omega_{T};\mathcal{C}_{\sharp}(\mathrm{Y}))$ 
	$$\lim_{\epsilon\rightarrow 0}\int_{\Omega_{T}}u^{\epsilon}(t,x)\phi(t,x,\frac{x}{\epsilon})\ \xdif x\xdif t=\int_{\Omega_{T}}\int_{Y}u^{0}(t,x,y)\phi(t,x,y)\ \xdif y\xdif x\xdif t.$$
	
\end{definition}
Fundamental  properties are the following.
\begin{proposition}\label{t-s}
	\begin{enumerate}
		\item From any \textcolor{blue}{uniformly (w.r.t. $\epsilon$)} bounded sequence $(u^\epsilon)$  in $L^2(\Omega_T)$, one can extract a subsequence that two-scale converges.
		\item\label{p2} If $(u^{\epsilon})$ is a bounded sequence in $L^{2}(\Omega_{T})$, which two-scale converges to $u^0\in L^{2}(\Omega_{T}\times Y)$, then 
		\begin{equation*}
			\liminf_{\epsilon \to 0}\|u^{\epsilon}\|_{L^{2}(\Omega_{T})}\geq\|u^{0}\|_{L^{2}(\Omega_{T}\times Y)}.
		\end{equation*}
		\item Let $(u^{\epsilon})$ be a bounded sequence in $H^{1}(\Omega_{T})$ which converges weakly to $u^{0}$ in $H^{1}(\Omega_{T})$. Then  $ u^\epsilon    \overset{2}{\rightharpoonup} u^0$ and there exists a function $u_{1}\in L^{2}(\Omega_{T};H^{1}_{\sharp}(Y))$ such that, up to a subsequence, $\nabla u^{\epsilon}\overset{2}{\rightharpoonup} \nabla_{x} u^{0}+\nabla_{y} u^{1}$.
		\item \label{pr}  Let $(u^{\epsilon})$ be a sequence in $L^{2}(\Omega_{T})$ such that $ u^\epsilon    \overset{2}{\rightharpoonup} u^0$ in $L^{2}(\Omega_{T}\times Y)$. Assume that
		$$\lim\limits_{\epsilon \rightarrow 0}\|u^{\epsilon}\|_{L^{2}(\Omega_{T})}=\|u^{0}\|_{L^{2}(\Omega_{T}\times Y)}.$$
		Then $ u^\epsilon$ strongly two-scale converges to $u^0$ \textcolor{blue}{(denoted in the sequel 
			$u^\epsilon\overset{2}{\to} u^0$)} in the following sense: for any sequence $(v^{\epsilon})\subset L^{2}(\Omega_{T})$  such that $ v^\epsilon    \overset{2}{\rightharpoonup} v^0$ in $L^{2}(\Omega_{T}\times Y)$, and for any bounded and admissible function $\phi$, one has:
		$$\!\! \lim\limits_{\epsilon \rightarrow 0}\int_{\Omega_{T}} \!\!\!\! u^{\epsilon}(t,x)v^{\epsilon}(t,x)\phi(t,x,\frac{x}{\epsilon})\ \xdif x\xdif t=\int_{\Omega_{T}\times Y} \!\!\!\!  u^{0}(t,x,y)v^{0}(t,x,y)\phi(t,x,y)\ \xdif y\xdif x\xdif t.$$
	\end{enumerate}
\end{proposition}
\subsection{Time-dependent two-scale decomposition}\label{2svi}

\textcolor{blue}{In this subsection, we introduce some notations that will be used to define two-scale convergence, following \cite{vin}. Let:}
$$
\begin{array}{ll}
	\hat{n}(x)= \max\{n \in \mathbb{Z} : n \leq x\}, \quad \hat{r}(x)= x - \hat{n}(x) \in [0, 1[ \quad \forall x \in \mathbb{R},\\[2ex]
	\mathcal{N}(x)= (\hat{n}(x_1),\hat{n}(x_2),\hat{n}(x_3)) \in \mathbb{Z}^3, \quad \mathcal{R}(x)= x - \mathcal{N}(x) \in Y \quad \forall x \in \mathbb{R}^3. 
\end{array}
$$
Thus, $x = \epsilon[\mathcal{N}(x/\epsilon) + \mathcal{R}(x/\epsilon)]$ for any $x \in \mathbb{R}^3$, $\epsilon>0$. 
The terms $\mathcal{N}(x/\epsilon)$ and $\mathcal{R}(x/\epsilon)$ can be interpreted as representing coarse-scale and fine-scale variables, respectively.
Set
\begin{equation}\label{trans} 
	\ds S^\epsilon(t,x, y) = \Bigl(t, \epsilon \mathcal{N}\Bigl(\frac{x}{\epsilon}\Bigr) + \epsilon y\Bigr), \quad \forall (t,x, y) \in (0,T) \times \mathbb{R}^3 \times Y . 
\end{equation} 
Definition \eqref{trans} is inspired by the dilation operator used in homogenization, as described in \cite{arbo,choquet1}. 
\textcolor{blue}{The operator $S^\epsilon$ is also called the unfolding operator (\cite{casa,ciora}).}
Notice  that $\ds S^\epsilon(t,x, y) = (t,x + \epsilon [y - \mathcal{R}(x/\epsilon)] )$ and
\begin{equation} \label{conu}
	\ds S^\epsilon(t,x, y) \to (t,x) \text{ uniformly in } (0,T) \times \mathbb{R}^3  \times Y \text{ as } \epsilon \to 0. 
\end{equation}
The transformation $S^\epsilon$  captures variations that occur on different scales. 
For any measurable function $\ds u: (0,T) \times \mathbb{R}^3  \times \mathbb{R}^3 \to \mathbb{R}$, where $y \mapsto u(t,x, y)$ is $Y$-periodic, one has 
\begin{equation}\label{tra} 
	\int_{0}^{T}\int_{\mathbb{R}^3} u\Bigl(t, x,\frac{x}{\epsilon}\Bigr)\  \xdif x  \xdif t = \int_{0}^{T}\int_{\mathbb{R}^3 \times Y} u(S^{\epsilon}(t,x, y),y)\  \xdif x  \xdif y \xdif t, 
\end{equation}
which allows reducing the two-scale convergence to the standard convergence in Lebesgue spaces.
The seminal result in \cite{vin} (also in \cite{casa,ciora}) is actually the equivalence
\begin{equation}
	u^\epsilon    \overset{2}{\rightharpoonup} u^0 \mbox{ in } L^2(\Omega_T\times Y) \Leftrightarrow  u^\epsilon \circ S^\epsilon \rightharpoonup u^0 \mbox{ weakly in } L^2(\Omega_T\times Y)  .
	\label{equ}
\end{equation}

\subsection{A variant of Aubin-Lions Lemma}\label{mouSs}

\textcolor{blue}{A classical tool for obtaining compactness results when dealing with evolution problems, is the Aubin-Lions argument. 
It does not apply here because  an estimate for a term in the form $\partial_t f^\epsilon$ does not {\it a priori} provide information about the temporal derivative of an extension of $f^\epsilon$.
And because we consider the Neumann problem on holes, we need corresponding extension theorems.
In this paper, we address this issue by establishing a kind of compensated compactness argument, that balances the behaviour of different types of extensions of the solutions. The first step is laid by the following result, adapted from \cite{moussa} by A. Moussa. }
We denote by $\mathcal{M}$ the space of Radon measures.
\begin{proposition} {\normalfont (Moussa \cite{moussa}, Prop. 3)}\label{mossa}
	Let $q \in [1, \infty]$, $ p \in [1, n)$, $ \alpha \in [1, np/(n-p))$, $m \in \mathbb{N}^*$.
	Set $r'=r/(r-1)$ if $r>1$ and $r'=\infty$ if $r=1$, $r=q,\alpha$.
	Assume that  $a_\epsilon$  and  $b_\epsilon$ are two sequences of functions weakly or weakly-$*$ convergent to $a$ and $b$ in  $L^q(0,T; W^{1,p}(\Omega))$ and $L^{q'}(0,T; L^{\alpha'}(\Omega))$, respectively.
	If $\dep_t b_\epsilon$ is bounded in $\mathcal{M}(0,T; H^{-m}(\Omega))$, 
	then
	\begin{equation}
		\ds \lim_{\epsilon \to 0} \int_{\Omega_T}a_\epsilon b_\epsilon\phi \xdif x \xdif t=  \int_{\Omega_T}a b\phi \xdif x \xdif t, \  \forall \phi \in \mathcal{C}((0,T)\times\overline{\Omega}).
	\end{equation}
\end{proposition}
%

\section{Main result: homogenized phase transition model}\label{sec3}

The main result of the paper is the following.

\begin{theorem}[Homogenized model]\label{main}
	Assume that $\M_{0}$  and $\theta^\epsilon_0 $ are uniformly bounded \textcolor{blue}{with regard to $\epsilon$}  in, respectively, $H^{1}(\Omega^{\epsilon})$ and $L^{2}(\Omega^{\epsilon})$ \textcolor{blue}{(\textit{i.e.} $\ds\Vert \M_{0}\Vert_{H^{1}(\Omega^{\epsilon})}\leq C$ and  $\ds\Vert \theta^\epsilon_0\Vert_{L^{2}(\Omega^{\epsilon})}\leq C$)}. 
	For  $\epsilon > 0$, let the triple $(\M, v^{\epsilon}, \varphi^{\epsilon})$ be a solution of \eqref{modelfit1}.
	Set  $v^{\epsilon}=F(\theta^{\epsilon})$ and $v^{\epsilon}_0=F(\theta^{\epsilon}_0)$, where $F$ is given by \eqref{fonction}. 
	There exist  extensions $\widetilde{\M}$, $\widetilde{\M_0}$, $\widetilde{v^{\epsilon}}$, $\widetilde{v^{\epsilon}_0}$ and $\widetilde{\varphi^{\epsilon}}$  of $\M$, $\M_0$, $v^{\epsilon}$, $v^\epsilon_0$  and  $\varphi^{\epsilon}$ such that, up to a subsequence \textcolor{blue}{(still denoted by $\epsilon$)}, the following two-scale convergence results  hold:
	\begin{eqnarray*}
		&& (\widetilde{\M}, \widetilde{v^{\epsilon}}, \widetilde{\varphi^{\epsilon}}) \overset{2}{\rightharpoonup}
		(\chi_{Y^\ast} \m, \chi_{Y^\ast} v, (\chi_{\mathbb{R}^{3} \setminus \overline{\Omega}} + \chi_{\Omega} \chi_{Y^\ast}) \varphi),
		\\
		&& {\color{blue}(\widetilde{\M}, \widetilde{v^{\epsilon}})\overset{2}{\to}(\chi_{Y^\ast} \m, \chi_{Y^\ast} v)}\\
		&& (\widetilde{\M_0}, \widetilde{v^{\epsilon}_0}) \overset{2}{\rightharpoonup}  (\chi_{Y^\ast} \m_0, \chi_{Y^\ast} v_0) ,
	\end{eqnarray*}
	with $\m_0 \in H^{1}(\Omega)$ and $v_0 \in L^2(\Omega_T,L^2_\sharp(Y))$. 
	The limit triple $(\m,v,\varphi) \in L^2(0,T;H^1(\Omega)) \times L^2(0,T;H^1(\Omega)) \times L^2(0,T;W_{bp}(\R^3))$ is a solution of the following problem, the effective problem corresponding to \eqref{modelfit}:
	\begin{subequations}\label{modelfit100I}
		\begin{flalign}
			&\gamma\dep_t \m=\dvg(\A^{\ast}\nabla\m)-\theta_{c}(|\m|^{2}-1)\m-\theta\m+\bar{\bs{\mu}}\nabla \varphi +\mathbf{H}_2\m \hspace{0.8cm} \mbox{in}~\Omega_{T},&
			\label{3.a}
			\\
			&c_{1}\dep_{t}\ln(\theta)+c_{2}\dep_t\theta-\m\cdot\dep_t\m=\dvg(\K^{\ast}\nabla(k_0 \ln(\theta) +k_1 \theta)) \hspace{1.4cm} \mbox{in}~\Omega_{T},&
			\label{3.b}
			\\
			&\dvg(\bs{\mu}^{\ast}\nabla \varphi + \mathbf{H}_1\m)=0\hspace{5.3cm} \mbox{on}~(0,T)\times\R^3 , &
			\label{3.c}
			\\
			&\A^{\ast}\nabla\m\cdot\nu^\epsilon=0  , \quad
			(\K^{\ast}\nabla (k_0\ln(\theta)+k_1\theta))\cdot\nu^\epsilon=0  \hspace{1.3cm} \mbox{on}~(0,T)\times\dep\Omega, &
			\label{3.d}
			\\
			& (\bs{\mu}^{\ast}\nabla \varphi + \mathbf{H}_1\m)\cdot\nu^\epsilon=0  \hspace{5.2cm} \mbox{on}~(0,T)\times\dep\Omega,&
			\label{3.e}
			\\
			& \theta(x,0)=\mathcal{M}_Y(\chi^\ast\theta_0) (x), \quad \m(x,0)=\overline\chi \m_0(x) \hspace{3.55cm} \mbox{in } \Omega,&
			\label{3.f}
		\end{flalign}
	\end{subequations}
	with $\theta = G(v)$,  $\A^\ast,\ \K^\ast,\ \bs{\mu}^\ast,\ \bar{\bs{\mu}},\ \mathbf{H}_1$ and $\mathbf{H}_2$
	are the 3 × 3 symmetric matrix whose entries,  for $1\leq i,j\leq 3$, are  defined by
	\begin{equation}
		\!\!\!\
		\begin{array}{ll}
			\ds A^{\ast}_{ij}=\mathcal{M}_{Y^\ast}\!\Bigl( A_{ij}+\sum_{k=1}^{3}A_{ik}\dep_{y_k}\omega_j\Bigr),\ K^{\ast}_{ij}=\mathcal{M}_{Y^\ast}\!\Bigl( K_{ij}+\sum_{k=1}^{3}K_{ik}\dep_{y_k}\widehat{\omega}_j\Bigr),\\
			\ds \mu^{\ast}_{ij}=\mathcal{M}_{Y}\!\Bigl((\chi_{\R^{3}\setminus\overline{\Omega}}+\chi_{\Omega}\chi_{Y^\ast}) \bigl(\mu_{ij}+\sum_{k=1}^{3}\mu_{ik}\dep_{y_k}\bar{\omega}^{1}_j\bigr)\Bigr),\ \bar{\mu}_{ij}=\delta_{ij}+\mathcal{M}_{Y^\ast}\!(\dep_{y_i}\bar{\omega}^{1}_j),\\
			\ds {H_1}_{ij}=\mathcal{M}_{Y}\!\Bigl(\chi_{Y^\ast}\Bigl(\delta_{ij} +\sum_{k=1}^{3}\mu_{ik}\dep_{y_k}\bar{\omega}^{2}_j\Bigr)\Bigr),\ {H_2}_{ij}=\delta_{ij} +\mathcal{M}_{Y^\ast}(\dep_{y_i}\bar{\omega}^{2}_j),
		\end{array}
	\end{equation}
	where $\delta_{ij}$ denotes the Kronecker delta, and the functions $\omega_i$, $\bar{\omega}^k_i$, and $\widehat{\omega}_i$ for $1 \leq i \leq 3$ and $k = 1,2$,  satisfy the auxiliary problems \eqref{echm}, \eqref{omega1}, \eqref{omega2} and \eqref{omegav}, defined below.
\end{theorem}

The structure of the homogenized problem remains essentially the same as that of the microscopic problem. The effective parameters, as is often the case after homogenization, are calculated from auxiliary problems defined on the unit cell.
An important point however is the definition of the effective magnetic field.
Observation of the last term of \eqref{3.a} reveals that the magnetic field at the macroscopic scale can be decomposed into two parts. It can be expressed as the sum of a field that, similar to the microscopic scale, derives from a scalar potential, and a linear function of the magnetization. 
Another way to view this result is to interpret it in light of the definition of the Curie temperature.
It has been widely investigated in dozens of works, most of them being devoted to measures in special class of materials.
Here, we can characterize it with explicit expression of the magnetic potential,
\textcolor{blue}{
$$W_{\text{mic}}(\vert \M \vert, \theta^\epsilon) = \int_{\Omega^\epsilon} \Bigl( \frac{1}{4} \theta_c \vert \M \vert^4 + \frac{1}{2}(\theta^\epsilon - \theta_c)\vert \M \vert^2 \Bigr) ,$$
 }
in the microscopic model \eqref{model1}, {\it versus}
\textcolor{blue}{$$
W_{\text{mac}}(\vert \mathbf{m} \vert, \theta) = \int_{\Omega} \Bigl( \frac{1}{4} \theta_c \vert \mathbf{m} \vert^4 + \frac{1}{2} \left( (\theta - \theta_c) I_3 - \mathbf{H}_2 \right) \mathbf{m} \cdot \mathbf{m} \Bigr)
$$
}
in the macroscopic model \eqref{3.a} (we denote by $I_3$ the identity matrix).
The first one has a global minimum at $\M = 0$ for all $\theta^\epsilon \ge \theta_c$, $\theta_c$ being the scalar Curie temperature at the microscale.
The same computation with $W_{\text{mac}}$ has to take into account the tensor $\mathbf{H}_2$ and thus leads to the definition of a Curie temperature tensor.
Let us point out that the result thus  enriches the understanding of the behavior of the material on the macroscopic scale.

\medskip

The rest of the paper is devoted to the proof of Theorem \ref{main}.

\section{Uniform estimates}\label{sec4}

The present section is devoted to the statement of uniform estimates for a solution $(\M,\theta^{\epsilon},\varphi^{\epsilon})$ of Problem \eqref{modelfit} and $v^{\epsilon}$ featured in \eqref{modelfit1}  under the hypothesis of Theorem \ref{main}. We first prove the following result.
\begin{lemma}\label{lem11}
	There exist real numbers $\eta_1$ and $\eta_2$ that do not depend on $\epsilon$ such that  energy estimate \eqref{energy-estimate} holds true. \textcolor{blue}{In particular, $\m^\epsilon$ is uniformly bounded in $L^\infty(0,T;L^4(\Omega^\epsilon))$:}
	\begin{equation}\label{L4}
		\ds \Vert \m^{\epsilon}\Vert_{L^\infty(0,T;L^4(\Omega^\epsilon))}\leq C.
	\end{equation}
\end{lemma}

\begin{proof} 
	The proof of \eqref{energy-estimate} under assumptions of Theorem \ref{main} is based on techniques  already employed in \cite{berti} and \cite{tilioua}. Indeed, 
	we multiply equations \eqref{1.a} by $\partial_t \M$ and \eqref{1.b} by $\theta^\epsilon$. Then, by an integration by parts and a summation of the two equations, we find:
	\begin{eqnarray}\label{dest}
		&&	\ds \gamma\int_{\Omega^{\epsilon}}\vert\dep_{t}\M\vert^2\ \xdif x +\int_{\Omega^{\epsilon}} \A^\epsilon \nabla \M \cdot \partial_t \nabla  \M \ \xdif x+\frac{\theta_{c}}{4}\frac{\xdif}{\xdif t}\int_{\Omega^{\epsilon}}\vert \M\vert^4\ \xdif x
		\nonumber \\
		&&	\qquad 
		\ds +\frac{c_2}{2}\frac{\xdif}{\xdif t}\int_{\Omega^{\epsilon}}\vert \theta^{\epsilon}\vert^2\ \xdif x
		+c_1\frac{\xdif}{\xdif t}\int_{\Omega^{\epsilon}}\theta^{\epsilon}\ \xdif x-\int_{\Omega^{\epsilon}} \nabla \varphi^{\epsilon}\cdot \dep_{t}\M\ \xdif x
		\nonumber \\
		&&	\qquad
		\ds +\int_{\Omega^{\epsilon}}\K^\epsilon\nabla(k_0\ln(\theta^\epsilon)+k_1\theta^\epsilon)\cdot\nabla\theta^\epsilon\ \xdif x
		=\theta_{c}\int_{\Omega^{\epsilon}}\M\cdot\dep_{t} \M\ \xdif x.
	\end{eqnarray}
	Next, bearing in mind the symmetry of $\A$, notice that 
	\begin{equation}\label{fb}
		\ds\int_{\Omega^{\epsilon}} \A^\epsilon \nabla \M \cdot \partial_t \nabla  \M \ \xdif x=\frac{1}{2}\frac{\xdif}{\xdif t}\int_{\Omega^{\epsilon}}\A^\epsilon \nabla \M \cdot  \nabla  \M \ \xdif x.
	\end{equation}
	A concern in \eqref{dest} is the term $\ds \int_{\Omega^{\epsilon}} \nabla \varphi^{\epsilon}\cdot \dep_{t}\M\ \xdif x$, which needs to be evaluated. To this aim, we refer to equation \eqref{1.c} recalled below:
	$$\dvg(\bs{\mu}^\epsilon\dep_{t}\nabla\varphi^{\epsilon} + \chi_{\Omega^\epsilon}\dep_{t}\M)=0 \ \mbox{ in the weak sense in }(0,T) \times\R^{3},$$
	Multiplying this equation by $\varphi^{\epsilon}$ and integrating by parts in $\R^3$, we obtain:
	\begin{equation}
		\ds \int_{\Omega^{\epsilon}} \nabla \varphi^{\epsilon}\cdot\dep_{t}\M\ \xdif x=-\int_{\R^3}\bs{\mu}^\epsilon\dep_{t}\nabla\varphi^{\epsilon}\cdot\nabla \varphi^{\epsilon}\ \xdif x,
		\label{demes}
	\end{equation}
	where, the permeability $\bs{\mu}^\epsilon$ being symmetric, $$\ds \int_{\Omega^{\epsilon}} \nabla \varphi^{\epsilon}\cdot\dep_{t}\M\ \xdif x=-\frac{1}{2}\frac{\xdif}{\xdif t}\int_{\R^3}\bs{\mu}^\epsilon\nabla\varphi^{\epsilon}\cdot\nabla \varphi^{\epsilon}\ \xdif x.$$
	Using Hölder's and Young's inequalities and the injection property of Lebesgue spaces $L^4\subset L^2$, the right-hand side of \eqref{dest} is bounded as follows:
	\begin{equation}\label{terme2}
		\ds\theta_{c}\int_{\Omega^{\epsilon}}\M\cdot\dep_{t} \M\ \xdif x\leq\frac{\theta_{c}^4}{8\gamma}+C_\Omega\Vert \M\Vert_{L^4(\Omega^\epsilon)}^{4}+\frac{\gamma}{2}\Vert \dep_{t}\M\Vert_{L^2(\Omega^\epsilon)}^{2}.
	\end{equation}
	Now, according to \eqref{fb}-\eqref{terme2} as well as the  positivity of $\theta^{\epsilon}$, we infer from \eqref{dest} that
	\begin{equation*}
		\frac{\xdif}{\xdif t}\mathcal{E}^{\epsilon}(t)	+\frac{\gamma}{2}\|\dep_t\M\|_{L^{2}(\Omega^{\epsilon})}^2+\ds\int_{\Omega^{\epsilon}}\K^{\epsilon}\nabla(k_0 \ln(\theta^{\epsilon}) +k_1 \theta^{\epsilon})\cdot\nabla\theta^{\epsilon} \ \xdif x\leq C_\Omega\mathcal{E}^{\epsilon}(t)+C,
	\end{equation*}
	where $C_\Omega$ and $C=\theta_{c}^4/8\gamma$ do not depend on $\epsilon$. Now, \textcolor{blue}{bearing in mind the coercivity assumptions on $\mathbf{A}$, $\mathbf{\mu}$ and $\mathbf{K}$}, a direct application of Grönwall's lemma allows  to obtain \eqref{energy-estimate}, with $\ds\eta_{1}=e^{C_\Omega T}>0$ and $\ds\eta_{2}=C(e^{C_\Omega T}-1)/C_\Omega>0$.  Finally, \eqref{L4} follows directly from the definition of $\mathcal{E}^{\epsilon}$. 
\textcolor{blue}{Notice that this  result gives sense to the definition of magnetic potentials given after Theorem \ref{main}.}
\end{proof}

Next we derive some estimates for the sequence $v^{\epsilon}$.
\begin{lemma}\label{lem2}
	There exists a constant $C>0$, such that
	\begin{equation}\label{enest1}
		\|v^\epsilon\|_{L^{2}(0,T;H^{1}(\Omega^{\epsilon}))} \leq C, \quad 
		\Vert\dep_{t}v^\epsilon\|_{L^{2}(0,T;H^{-1}(\Omega^{\epsilon}))} \leq C.
	\end{equation}
\end{lemma}
\begin{proof} Multiply \eqref{2.b} by $v^{\epsilon}$. Integration by parts gives
	\begin{equation}\label{1.0}
		\ds \frac{1}{2}\frac{\xdif}{\xdif t}\int_{\Omega^{\epsilon}}\vert v^{\epsilon}\vert^{2}\xdif x+\int_{\Omega^{\epsilon}}\K^{\epsilon}(x)g(\theta^\epsilon)\nabla v^{\epsilon}\cdot \nabla v^{\epsilon} \xdif x
		=\ds\int_{\Omega^{\epsilon}}(\M\cdot\dep_{t}\M)v^{\epsilon} \xdif x,
	\end{equation}
	where $g$, \textcolor{blue}{defined by \eqref{tra20}, is bounded by parameters depending on the model chosen for the heat conductivity and the specific heat (see \eqref{heatlaw} where $c_1,c_2,k_0,k_1>0$):}
	\begin{equation}\label{0}
		0<\pi_{0}=\min\Big(\frac{k_{0}}{c_{1}},\frac{k_{1}}{c_{2}}\Big)\leq g(x)\leq\max\Big(\frac{k_{0}}{c_{1}},\frac{k_{1}}{c_{2}}\Big)=\pi_{1},\ \forall x>0.
	\end{equation}
	Using H\"older's and Young's inequalities, we get
	\begin{equation}\label{1}
		\ds\int_{\Omega^{\epsilon}}(\M\cdot\dep_{t}\M)v^{\epsilon} \xdif x
		\leq 
		\Vert \M \partial_t \M \Vert_{L^{3/2}(\Omega^\epsilon)}\Vert v^{\epsilon}\Vert_{L^3(\Omega^\epsilon)}
		\leq
		\frac{C(t)}{2} + \frac{1}{2} \Vert v^\epsilon \Vert^2_{L^3(\Omega^\epsilon)},
	\end{equation}
	where $t \mapsto C(t)$ is uniformly bounded in  $L^1(0,T)$ since $ L^\infty(0,T;H^1(\Omega^\epsilon)) \subset L^\infty(0,T;L^6(\Omega^\epsilon))$  (see Lemma \ref{lem11}).
	Next, the interpolation inequality and once again the Young inequality leads to
	$$\ds  \frac{1}{2} \Vert v^{\epsilon}\Vert_{L^3(\Omega^\epsilon)}^2 \leq  \frac{C}{2} {\color{blue}\Vert  v^{\epsilon}\Vert_{H^1(\Omega^\epsilon)}}\Vert v^{\epsilon}\Vert_{L^2(\Omega^\epsilon)}
	\leq \left(\frac{C^2}{8\pi_0 \alpha_2}+{\color{blue}\frac{\pi_0 \alpha_2}{2}} \right)\Vert v^{\epsilon}\Vert_{L^2(\Omega^\epsilon)}^2 + \frac{\pi_0 \alpha_2}{2} \Vert \nabla v^{\epsilon}\Vert_{L^2(\Omega^\epsilon)}^2.
	$$
	Inserting the latter result into \eqref{1} and bearing in mind the coercivity of $\K^{\epsilon}$, we deduce from \eqref{1.0}, after integrating over time, that
	\begin{eqnarray*}
		&& \frac{1}{2} \int_0^\tau \frac{\xdif}{\xdif t}\int_{\Omega^{\epsilon}}\vert v^{\epsilon}\vert^{2}
		+ \frac{\pi_0 \alpha_2}{2} \int_0^\tau \int_{\Omega^\epsilon} \vert \nabla v^\epsilon \vert^2 
		\leq \frac{C(t)}{2} + {\color{blue}\left(\frac{C^2}{8\pi_0 \alpha_2}+\frac{\pi_0 \alpha_2}{2} \right)}\int_0^\tau \int_{\Omega^\epsilon} \vert v^{\epsilon}\vert^{2}
	\end{eqnarray*}
	for any $\tau \in (0,T)$. According to Gr$\ddot{\mbox{o}}$nwall's lemma, the following uniform estimates follow:
	\begin{equation*}\label{es}
		\begin{array}{ll}
			\Vert v^{\epsilon}\Vert_{L^{\infty}(0,T;L^2(\Omega^\epsilon))}\leq C, \quad 
			\Vert  v^{\epsilon}\Vert_{L^{2}(0,T;H^1(\Omega^\epsilon))}\leq C   .
		\end{array}
	\end{equation*}

	For the estimate of  $\dep_{t} v^\epsilon$ in $L^{2}(0,T;H^{-1}(\Omega^{\epsilon}))$, it is sufficient to show that there exists a constant $C$ such that,
	\begin{equation}\label{V1}
		\Bigr \vert \int_{0}^{T}\int_{\Omega^{\epsilon}} \dep_{t} v^\epsilon \Phi\xdif x\xdif t \Bigr \vert\leq C \Vert \Phi \Vert_{L^{2}(0,T;H_0^1(\Omega^\epsilon))},\ \forall \Phi\in L^{2}(0,T;H_0^1(\Omega^\epsilon)).
	\end{equation}
	Multiplying  \eqref{2.b} by $\Phi$ and integrating over $\Omega^\epsilon_T$, one obtains
	\begin{eqnarray*}
		&&	\int_{\Omega^{\epsilon}_T} \dep_{t} v^\epsilon \Phi\xdif x\xdif t
		=-\ds \int_{\Omega^{\epsilon}_T}\K^{\epsilon}(x) g(\theta^\epsilon)\nabla v^{\epsilon}\cdot \nabla \Phi \xdif x \xdif t+\int_{\Omega_T^{\epsilon}}(\M\cdot\dep_{t}\M)\Phi \xdif x \xdif t.
		\\
		&& \qquad \qquad
		\le C \Vert \nabla \Phi \Vert_{L^2(\Omega^\epsilon_T)} + C \Vert \Phi \Vert_{L^2(0,T;L^3(\Omega^\epsilon))}  .
	\end{eqnarray*}
	Using once again an interpolation inequality together with Poincar\'e's inequality, we get $\ds\Vert \Phi \Vert_{L^2(0,T;L^3(\Omega^\epsilon))}\le C \biggr(\int_0^T \Vert \nabla \Phi \Vert_{L^2(\Omega^\epsilon)} \Vert \Phi \Vert_{L^2(\Omega^\epsilon)}  \biggr)^{1/2} \le C \Vert \Phi \Vert_{L^2(0,T; H_0^1(\Omega^\epsilon)   )}$. 
	Lemma \ref{lem2} is proved.
\end{proof}

\section{Homogenization process}\label{sec5}

In this section, we homogenize the phase transition micro-model in perforated domains by letting $\epsilon$ tend to zero in Problem \eqref{modelfit1}.
The first two subsections are devoted to the preliminary statement of convergence results for solutions of \eqref{modelfit1}: weak and strong two-scale convergence results are respectively given  in paragraphs \ref{2s} and \ref{2st}, strong  results being proved using a compensated compactness argument;  these results are re-interpreted, also in paragraph \ref{2st}, using an adaptation of the two-scale decomposition of Visintin \cite{vin}.
Finally, we let $\epsilon\to 0$ in a variational formulation of \eqref{modelfit1} in paragraph \ref{passage} using 
especially a Vitali compactness argument  to overcome the inapplicability of the classical Lebesgue theorem.

\subsection{Two-scale convergence results}\label{2s}
Since the structure of the perforated domain $\Omega^\epsilon$ oscillates with $\epsilon$, \textcolor{blue}{and because we consider the Neumann problem on the holes (see \eqref{2.d} and \eqref{2.e})}, we first need to extend the sequences $\mathbf{m}^\epsilon$, $v^{\epsilon}$, and $\varphi^{\epsilon}$ to the whole domain $\Omega$. This requires the use of a suitable extension operator, as presented in the following lemma:

\begin{lemma}{\normalfont (Damlamian-Donato \cite{damlamain-donato})}\label{ext}
	There exists $c>0$ such that, for all $\epsilon> 0$, there exists an extension operator $P^{\epsilon}$ from $H^{1}(\Omega^{\epsilon})$ to $H^{1}(\Omega)$  such that
	$$P^{\epsilon}\psi=\psi \   \mbox{in } \Omega^{\epsilon} \mbox{ and }
	\|P^{\epsilon}\psi\|_{H^{1}(\Omega)}\leq c\|\psi\|_{H^{1}(\Omega^{\epsilon})}  \quad  \forall \psi \in H^{1}(\Omega^{\epsilon}). $$
\end{lemma}

According to  Lemmas \ref{lem11} and \ref{lem2},  the following uniform estimates hold for the extended solutions:
\begin{equation}\label{ES1}
	\begin{array}{ll}
		\ds	\|P^{\epsilon}\M\|_{L^{2}(0,T;H^{1}(\Omega))}\leq C, \\
		\ds\|\chi_{\R^{3}\setminus\overline{\Omega}}\nabla\varphi^{\epsilon}+\chi_{\Omega}\nabla P^{\epsilon}\varphi^{\epsilon}\|_{L^{\infty}(0,T;L^{2}(\R^{3}))}\leq C, \\
		\ds\|P^{\epsilon}v^{\epsilon}\|_{L^{2}(0,T;H^1(\Omega))}\leq C.
	\end{array}
\end{equation}
The latter estimates with Proposition \ref{t-s} let us claim the following  results.
\begin{proposition}\label{rescon}
	There exist limit functions  $\m \in L^2(0,T; H^1(\Omega))$, $\m_1 \in L^2(\Omega_T; H^1_{\sharp}(Y))$, $v \in L^{2}(0,T;H^1(\Omega))$ and $v_1 \in L^{2}(\Omega_T;H^1_{\sharp}(Y))$ such that the following convergence hold true,  up to a subsequence (not relabeled for simplicity):
	\begin{equation}\label{conv}
		\begin{array}{ll}
			P^\epsilon \M \rightharpoonup \m \quad \text{weakly  in } L^2(0,T; H^1(\Omega)), \\
			P^\epsilon v \rightharpoonup v \quad \text{weakly in } L^2(0,T; H^1(\Omega)), \\
			\nabla(P^\epsilon \M)  \overset{2}{\rightharpoonup} \nabla \m + \nabla_y\m_1 \quad \text{ in } L^2(\Omega_T\times Y), \\
			\nabla(P^\epsilon v^\epsilon)  \overset{2}{\rightharpoonup} \nabla v + \nabla_y v_1 \quad \text{in } L^2(\Omega_T\times Y).
		\end{array}
	\end{equation}
\end{proposition}

The convergence results for the potential field are described in the following proposition.
\begin{proposition}
	Let $\widetilde{\varphi^\epsilon }$ be the extension of $\varphi^\epsilon$ defined by 
	$$ \widetilde{\varphi^\epsilon }= \chi_{\mathbb{R}^{3} \setminus \overline{\Omega}} \varphi^\epsilon + \chi_{\Omega} P^\epsilon \varphi^\epsilon .$$
	There exist $\varphi \in L^{\infty}(0,T; W_{bp}(\mathbb{R}^{3}))$ and $\varphi_1 \in L^{\infty}(0,T; L^2(\mathbb{R}^{3}; H^1_{\sharp}(Y)))$,   \textcolor{blue}{such that up to a subsequence, the following convergence results hold:}
	\begin{equation}\label{conv1}
		\begin{aligned}
			&\widetilde{\varphi^\epsilon } \rightharpoonup \varphi \quad \text{weakly in } L^2(0,T; W_{bp}(\mathbb{R}^{3})),\\
			&\nabla\widetilde{\varphi^\epsilon } \overset{2}{\rightharpoonup} \nabla \varphi + \nabla_y \varphi_1 \quad \text{in  } L^2((0,T)\times\mathbb{R}^{3}  \times Y).
		\end{aligned}
	\end{equation}
\end{proposition}
\begin{proof} For the proof, it is sufficient to observe that 
	$$
	\Vert\widetilde{\varphi^\epsilon }\Vert_{L^{2}(0,T; W_{bp}(\mathbb{R}^{3}))} \leq C,
	$$ 
	which can be verified using the extension property and the estimate $\eqref{ES1}_3$. 
\end{proof}

In the remainder of this paper, we denote by  $\widetilde{v^\epsilon}$ and $\widetilde{\M}$ the extensions by zero of $v^\epsilon$ and $\M$ in the holes of $\Omega^\epsilon$. Notice that
$$\widetilde{v^\epsilon}=\chi_{\Omega^\epsilon} P^{\epsilon}v^\epsilon\ \mbox{and}\  \widetilde{\M}=\chi_{\Omega^\epsilon} P^{\epsilon}\M.$$
One checks (see Lemma 2.3 and Remark 2.4 in \cite{allaire1} for details) that
\begin{equation}\label{conv2}
	\widetilde{v^\epsilon}\rightharpoonup \bar{\chi} v \ \mbox{and}\ \widetilde{\M}\rightharpoonup \bar{\chi} \m\ \mbox{weakly in } L^2(\Omega_{T}).
\end{equation}

\medskip

Let us end the subsection with a result for the initial data. 
Assume as in Theorem \ref{main} that $\M_{0}$  and $\theta^\epsilon_0 $ are uniformly bounded in, respectively, $H^{1}(\Omega^{\epsilon})$ and $L^{2}(\Omega^{\epsilon})$. 
Define $\widetilde{v_0^\epsilon}$ and $\widetilde{\M_0}$ by
$$\widetilde{v_0^\epsilon}=\chi_{\Omega^\epsilon} P^{\epsilon}v_0^\epsilon\ \mbox{and}\  \widetilde{\M_0}=\chi_{\Omega^\epsilon} P^{\epsilon}\M_0.$$
With this definition, which is similar to the one used for $(\widetilde{v^\epsilon},\widetilde{\M})$, the existence of limit functions  $\m_0 \in H^{1}(\Omega)$ and $v_0 \in L^2(\Omega_T;L^2_\sharp(Y))$ such that
\begin{equation}
	(\widetilde{\M_0}, \widetilde{v^{\epsilon}_0}) \overset{2}{\rightharpoonup}  (\chi_{Y^\ast} \m_0, \chi_{Y^\ast} v_0) 
	\label{convinit}
\end{equation}
is obvious.

\subsection{Strong two-scale convergence results}\label{2st}

Weak convergence results \eqref{conv}-\eqref{conv2} are not sufficient for passing to the limit in the nonlinear terms of Problem  \eqref{modelfit1}. 
A classical tool for obtaining compactness results when dealing with evolution problems, is the Aubin-Lions argument. 
It does not apply here because  an estimate for a term in the form $\partial_t f^\epsilon$ does not  provide information about $\partial_t (P^\epsilon f^\epsilon)$. 
\textcolor{blue}{To address this problem, we propose a kind of compensated compactness argument for passing to the limit in the product of the two kind of extensions used in the paper (extension by zero and extension by operator $P^\epsilon$). The following lemma, obtained using Proposition \ref{mossa}, is the first step toward a strong two-scale convergence result.}

\begin{lemma}\label{cor1}
	The following convergence results hold true:
	\begin{eqnarray}
		\label{cf2}
		&&		\ds \lim_{\epsilon \to 0} \int_{\Omega_{T}}  (P^{\epsilon}v^\epsilon) \widetilde{v^\epsilon}\phi \xdif x \xdif t = \int_{\Omega_{T}} \bar{\chi}v^2 \phi \xdif x \xdif t ,\  \forall \phi \in \mathcal{C}((0,T)\times\overline{\Omega}),
		\\
		\label{cvM}
		&&
		\lim_{\epsilon \to 0} \int_{\Omega_{T}}  (P^{\epsilon}\M) \widetilde{\M}\mpsi \xdif x \xdif t = \int_{\Omega_{T}} \bar{\chi}\m^2 \mpsi \xdif x \xdif t,\  \forall \mpsi \in \mathcal{C}((0,T)\times\overline{\Omega}).
	\end{eqnarray}
\end{lemma}
\begin{proof} In Prop. \ref{mossa}, set $n = 3$, $q = 2$, $p = 2$. Consequently, $\alpha = 2<np/(n-p)$.
	Set $a_{\epsilon} = P^{\epsilon} v^{\epsilon}$ and $b_{\epsilon} = \widetilde{v^{\epsilon}}$. 
	It follows  from $\eqref{ES1}_{4,5}$ that $a_{\epsilon}$, $b_{\epsilon}$ and $\partial_t b_{\epsilon}$ are respectively bounded in $L^{2}(0, T; H^1(\Omega))$, $ L^{2}(0, T; L^{2}(\Omega))$, and $L^{2}(0, T; H^{-1}(\Omega))$. 
	Using the convergences $\eqref{conv}_{2}$ and \eqref{conv2} as well as the previous proposition, one obtains \eqref{cf2}.
	To obtain \eqref{cvM}, it is sufficient to consider $ a_{\epsilon} = P^{\epsilon} \M \in L^{2}(0, T; H^{1}(\Omega))$ and $b_{\epsilon} = \widetilde{\M} \in L^{2}(0, T; L^{2}(\Omega))$, with $\partial_{t} b_{\epsilon} \in L^{2}(0, T; L^{2}(\Omega)) $.  
\end{proof}

Next, a strong  two-scale  convergence result for  $\widetilde{v^{\epsilon}}$ follows from Lemma \ref{cor1}. 
\begin{lemma}\label{lem4}
	The sequences  $\widetilde{v^{\epsilon}}$ and $\widetilde{\M} $ strongly  two-scale converge  in $L^2(\Omega_{T} \times Y)$ to $\chi_{Y^\ast} v$ and $\chi_{Y^\ast} \m$, respectively,  in the sense of Proposition \ref{t-s}.
\end{lemma}
\begin{proof} Let $f^\epsilon$ be a sequence in $L^2(\Omega_T)$ and $f \in L^2(\Omega_T;L_{\sharp}^2(Y))$ such that $f^\epsilon \overset{2}{\rightharpoonup} f$ in $L^2(\Omega_T \times Y)$.
	The aim is to pass to the limit $\epsilon \to 0$ in 
	\begin{equation}\label{cf3}
		\ds \int_{\Omega_{T}} \widetilde{v^\epsilon}f^\epsilon \xdif x\xdif t = \int_{\Omega_{T}} \chi_{\Omega^\epsilon}vf^\epsilon\xdif x\xdif t+\int_{\Omega_{T}} (\widetilde{v^\epsilon}-\chi_{\Omega^\epsilon}v)f^\epsilon \xdif x\xdif t
	\end{equation}
	where 
	\begin{eqnarray*}
		&& \Bigr\vert\int_{\Omega_{T}} (\widetilde{v^\epsilon}-\chi_{\Omega^\epsilon}v)f^\epsilon\xdif x\xdif t\Bigr\vert\leq C \Vert\widetilde{v^\epsilon}-\chi_{\Omega^\epsilon}v\Vert_{L^2(\Omega_{T})} ,
		\\
		&& \ds\Vert\widetilde{v^\epsilon}-\chi_{\Omega^\epsilon}v\Vert_{L^2(\Omega_{T})}^{2}= \int_{\Omega_{T}}\widetilde{v^\epsilon}^2\xdif x\xdif t+\int_{\Omega_{T}}\chi_{\Omega^\epsilon}v^2\xdif x\xdif t-2\int_{\Omega_{T}}\chi_{\Omega^\epsilon}\widetilde{v^\epsilon}v\xdif x\xdif t .
	\end{eqnarray*}
	Using  $(\chi_{\Omega^\epsilon})^2 = \chi_{\Omega^\epsilon}$ and  \eqref{cf2} with $\phi = 1$, one obtains
	$$
	\ds \lim_{\epsilon\rightarrow 0} \ds\Vert\widetilde{v^\epsilon}-\chi_{\Omega^\epsilon}v\Vert_{L^2(\Omega_{T})}=0.
	$$
	Therefore,  the limit in \eqref{cf3} reads
	$$ \lim_{\epsilon\rightarrow 0} \ds \int_{\Omega_{T}} \widetilde{v^\epsilon}f^\epsilon \xdif x\xdif t = \int_{\Omega_{T}}\int_{Y} \chi_{Y^\ast}vf\xdif x\xdif y \xdif t .$$ 
	The proof of the convergence result  for $\M$ follows the same line.  
\end{proof}

%
%
Passing to the limit in the nonlinear terms of Problem \eqref{modelfit1} requires further a slightly different formalization of the two-scale convergence method. 
Here we use the dilation operator introduced in Subsection \ref{2svi}.
In particular, it allows to interpret the former two-scale convergence results  as follows.
Notice that, for using the transform $S^\epsilon$, we extend by zero in $\mathbb{R}^3\setminus \Omega$ all functions defined in $\Omega$ without changing the notations.

\begin{lemma}\label{strong1}
	\textcolor{blue}{Up to a subsequence}, the following convergence results hold true
	\begin{equation}\label{CON}
		\begin{array}{ll}
			\ds \widetilde{v^\epsilon} \circ S^{\epsilon} \to \chi_{Y^\ast} v \ \mbox{ in }L^2(\Omega_T \times Y )\\
			\ds \widetilde{\M}\circ S^{\epsilon}  \to \chi_{Y^\ast} \m \ \mbox{ in }L^2(\Omega_T \times Y ),\\
			\ds \dep_t\widetilde{\M}\circ S^{\epsilon} \rightharpoonup \chi_{Y^\ast} \dep_t\m \ \mbox{weakly in }L^2(\Omega_T \times Y ).
		\end{array}
	\end{equation}
	\begin{equation}\label{correc}
		\begin{array}{ll}
			\ds \nabla P^\epsilon v^\epsilon\circ S^\epsilon\rightharpoonup \nabla_x v+\nabla_y v_1\ \mbox{weakly in } \ L^2(\Omega_{T}\times Y),\\
			\ds \nabla P^\epsilon \M\circ S^\epsilon\rightharpoonup \nabla_x \m+\nabla_y \m_1\ \mbox{weakly in } \ L^2(\Omega_{T}\times Y),\\
			\ds \nabla \widetilde{\varphi^\epsilon }\circ S^\epsilon\rightharpoonup \nabla_x \varphi+\nabla_y \varphi_1\ \mbox{weakly in } \ L^2((0,T)\times\R^3\times Y).
		\end{array}
	\end{equation}
\end{lemma}
\begin{proof}  The first convergence in \eqref{CON} means
	\begin{equation}\label{limt}
		\ds\lim_{\epsilon\rightarrow 0} \Vert \widetilde{v^\epsilon}(S^{\epsilon}(t,x,y))-\chi_{Y^\ast}(y) v(t,x)\Vert_{L^2(\Omega_{T}\times Y)} =0.
	\end{equation}
	From \eqref{tra}, one has
	%
	$$	\ds        \int_{0}^{T}\int_{\Omega}\int_{Y}  \vert \widetilde{v^\epsilon}(S^{\epsilon}(t,x,y))\vert^2\xdif x\xdif y\xdif t
	= \int_{0}^{T}\int_{\Omega}  \vert \widetilde{v^\epsilon}(t,x)\vert^2\xdif x\xdif t   , $$
	%
	where, according to Lemma \ref{lem4}, 
	\begin{equation*}
		\ds\lim_{\epsilon\rightarrow 0} \int_{0}^{T}\int_{\Omega}  \vert \widetilde{v^\epsilon}(t,x)\vert^2\xdif x\xdif t=\int_{0}^{T}\int_{\Omega}\int_{Y}  \vert \chi_{Y^\ast}(y)v(t,x)\vert^2\xdif x\xdif y\xdif t  .
	\end{equation*}
	One infers from the two latter relations that
	\begin{equation}\label{tra4}
		\ds \lim_{\epsilon\rightarrow 0} \Vert \widetilde{v^\epsilon}(S^{\epsilon}(t,x,y))\Vert^{2}_{L^2(\Omega_{T}\times Y)} = \Vert \chi_{Y^\ast}(y) v(t,x)\Vert^{2}_{L^2(\Omega_{T}\times Y)}.
	\end{equation}
	With \eqref{tra4} in hand, \eqref{limt} follows from the weak convergence in $L^2(\Omega_T \times Y)$ of $\widetilde{v^\epsilon} \circ S^{\epsilon} $ to $\chi_{Y^\ast} v$ which is ensured by Lemma \ref{lem4} and \eqref{equ}. 
	The proof of the convergence of $\widetilde \M \circ S^\epsilon$ in $L^2(\Omega_T \times Y)$ follows the same lines. The convergence of $\dep_t\widetilde{\M}\circ S^{\epsilon}$ is a direct consequence of Lemma \ref{lem11} that ensures the uniform boundedness of $\dep_t\M$ in $L^2(\Omega_T^\epsilon)$  and thus of $\dep_t\widetilde{\M}$ in $L^2(\Omega_T)$. 
	Convergence results \eqref{correc} for the gradients follows from $\eqref{conv}_{4,5}$, $\eqref{conv1}_{2}$ and \eqref{equ}.
\end{proof} 

\bigskip

We now present a result for handling nonlinear terms, particularly those involving a function composed with a Lipschitz but unbounded function which prevents from using Lebesgue's dominated convergence theorem.
An example in Problem \eqref{modelfit1} is  function $G$.
\textcolor{blue}{The following result of course exploits the two-scale compactness established for $\widetilde{v^{\epsilon}}$ in Lemma \ref{lem4}.}
\begin{proposition}\label{prop}
	For every Lipschitz-continuous function $h : \mathbb{R} \to \mathbb{R}$, we have, up to a subsequence:
	\begin{equation}\label{strong}
		h(\widetilde{v^\epsilon}) \overset{2}{\to} h(\chi_{Y^\ast} v) \ \text{ strongly two-scale in} \ L^2(\Omega_{T}\times Y).
	\end{equation}
\end{proposition}
\begin{proof}	To prove \eqref{strong}, it is sufficient to show that
	$$	h(\widetilde{v^\epsilon}\circ S_\epsilon) \to h(\chi_{Y^\ast} v) \ \text{  in} \ L^2(\Omega_{T}\times Y).$$
	Since $h$ is a Lipschitz continuous function, there exists a constant $C_h$ such that
	\begin{equation*}
		\ds\Vert h(\widetilde{v^\epsilon}\circ S_\epsilon)-h(\chi_{Y^\ast} v) \Vert_{L^2(\Omega_{T}\times Y)}\leq C_h
		\ds\Vert \widetilde{v^\epsilon}\circ S_\epsilon-\chi_{Y^\ast} v \Vert_{L^2(\Omega_{T}\times Y)}.
	\end{equation*}
	Now, using Lemma \ref{strong1}, one gets
	\begin{equation*}
		\ds\lim_{\epsilon\to 0}\Vert h(\widetilde{v^\epsilon}\circ S_\epsilon)-h(\chi_{Y^\ast} v) \Vert_{L^2(\Omega_{T}\times Y)}=0,
	\end{equation*}
	that is \eqref{strong}. The proposition is proved.
\end{proof}

\begin{corollary}\label{cor}
	If $h : \mathbb{R} \to \mathbb{R}$ is a Lipschitz continuous function, then, \textcolor{blue}{up to a subsequence:}
	\begin{equation}\label{strong2}
		\chi_{\Omega^\epsilon}h(\widetilde{v^\epsilon}) \overset{2}{\to} \chi_{Y^\ast}h( v)\text{ strongly two-scale in} \ L^2(\Omega_{T}\times Y).
	\end{equation}
\end{corollary}
\begin{proof} 
	First, notice  that (see Corollary 1.2. of \cite{vin})
	\begin{equation}\label{47}
		\ds\chi_{\Omega^\epsilon}(x)=\chi\left(\frac{x}{\epsilon}\right)=\chi\left(\mathcal{R}\left(\frac{x}{\epsilon}\right)\right) .
	\end{equation}
	Hence proving \eqref{strong2} is equivalent to proving that
	\begin{equation*}
		\chi^\ast(y) h(\widetilde{v^\epsilon}(S_\epsilon(t,x,y)) \to \chi_{Y^\ast}(y)h( v(t,x,y)) \ \text{ in} \ L^2(\Omega_{T}\times Y).	
	\end{equation*}
	The latter results from using Proposition \ref{prop}  in the following decomposition  
	\begin{eqnarray*}
		&&		\ds\Vert \chi_{Y^\ast}h(\widetilde{v^\epsilon}\circ S_\epsilon)-\chi_{Y^\ast}h(v) \Vert_{L^2(\Omega_{T}\times Y)}\leq \Vert \chi_{Y^\ast}\Vert_{L^\infty(Y)}
		\ds\Vert h(\widetilde{v^\epsilon}\circ S_\epsilon)-h(\chi_{Y^\ast} v) \Vert_{L^2(\Omega_{T}\times Y)}
		\nonumber \\
		&& \qquad \qquad \quad
		+\Vert \chi_{Y^\ast}h(\chi_{Y^\ast}v)-\chi_{Y^\ast}h(v) \Vert_{L^2(\Omega_{T}\times Y)},
	\end{eqnarray*}
	where $\chi_{Y^\ast}h(\chi_{Y^\ast}v)=\chi_{Y^\ast}h(v)$.
\end{proof}
\subsection{Derivation of the homogenized problem}\label{passage}
In this subsection, we finally derive the homogenized system associated with \eqref{modelfit1}. We aim at proving that the limit functions defined in Subsection \ref{2s}  are ruled by the following problem, which  is actually a reformulation of Theorem \ref{main}:
\begin{equation}\label{modelfit100}
	\left\{\begin{array}{ll}
		\gamma\dep_t \m=\dvg(\A^{\ast}\nabla\m)-\theta_{c}(|\m|^{2}-1)\m-G(v)\m+\bar{\bs{\mu}}\nabla \varphi +\mathbf{H}_2\m&\mbox{in}~\Omega_{T}\\
		\dep_t v-\m\cdot\dep_t\m=\dvg(\K^{\ast}g\circ G(v)\nabla v))&\mbox{in}~\Omega_{T},\\
		\dvg(\bs{\mu}^{\ast}\nabla \varphi + \mathbf{H}_1\m)=0&\mbox{on}~(0,T)\times \R^3, \\
		\A^{\ast}\nabla\m\cdot\nu=0    &\mbox{on}~(0,T)\times\dep\Omega,  \\
		(\K^{\ast}g\circ G(v)\nabla v)\cdot\nu=0   &\mbox{on}~(0,T)\times\dep\Omega , \\
		(\bs{\mu}^{\ast}\nabla \varphi + \mathbf{H}_1\m)\cdot\nu=0   &\mbox{on}~(0,T)\times\dep\Omega,\\
		v(x,0)=\mathcal{M}_Y(\chi^\ast v_{0}(x)), \quad \m(x,0)=\overline\chi \m_{0}(x) &\mbox{in}~ \Omega, \\
	\end{array}\right.
\end{equation}
with $\A^\ast,\ \K^\ast,\ \bs{\mu}^\ast,\ \bar{\bs{\mu}},\ \mathbf{H}_1$ and $\mathbf{H}_2$
defined in Theorem \ref{main}.

\begin{proof} 
	We multiply the first and second equations of $(\ref{modelfit1})$ by $r^\epsilon(t,x)=r(t,x,x/\epsilon)$ and $\p^\epsilon(t,x)=p(t,x,x/\epsilon)$, respectively, where the test functions $r$ and $\p$ belong to $C^{\infty}(\overline{\Omega_{T}}) \otimes C^{\infty}_\sharp(Y)$. Additionally, we multiply the third equation of $(\ref{modelfit1})$ by $\Psi^\epsilon \in \mathcal{D}((0,T)\times\R^3) \otimes C^{\infty}_\sharp(Y)$. 
	Then,  integrating by parts and considering that
	\begin{eqnarray*}
		&& \nabla \p^\epsilon(t,x)=\nabla_x \p(t,x,x/\epsilon)+\epsilon^{-1}\nabla_y \p(t,x,x/\epsilon),\\
		&& \nabla r^\epsilon(t,x)=\nabla_x r(t,x,x/\epsilon)+\epsilon^{-1}\nabla_y r(t,x,x/\epsilon),\\
		&& \nabla \Psi^\epsilon(t,x)=\nabla_x \Psi(t,x,x/\epsilon)+\epsilon^{-1}\nabla_y \Psi(t,x,x/\epsilon),\\
	\end{eqnarray*}
	one gets 
	\begin{eqnarray*}
		&& \int_{\Omega_{T}}(\gamma\dep_t \M+\theta_{c}(|\M|^{2}-1)\M+G(v^{\epsilon})\M-\nabla \varphi^{\epsilon})\cdot\p^\epsilon \xdif x\xdif t
		\nonumber \\
		&& \qquad\qquad
		+\int_{\Omega^{\epsilon}_{T}}\A^{\epsilon}\nabla\M\cdot\nabla_{x} \p^\epsilon\xdif x\xdif t+\ds\frac{1}{\epsilon}\int_{\Omega^{\epsilon}_{T}}\A^{\epsilon}\nabla\M\cdot\nabla_{y} \p^\epsilon \xdif x\xdif t=0,
		\\
		&& \int_{\Omega^{\epsilon}_{T}}(\dep_{t}v^{\epsilon}-
		\M\cdot\dep_{t}\M)r^\epsilon \xdif x\xdif t+\ds\int_{\Omega^{\epsilon}_{T}}\K^{\epsilon}g(G(v^{\epsilon}))\nabla v^{\epsilon}\cdot\nabla_{x} r^\epsilon \xdif x\xdif t
		\nonumber \\
		&& \qquad\qquad
		+\frac{1}{\epsilon}\int_{\Omega^{\epsilon}_{T}}\K^{\epsilon}g(G(v^{\epsilon}))\nabla v^{\epsilon}\cdot\nabla_{y} r^\epsilon \xdif x\xdif t=0,
		\\
		&& \int_{\R^3\times(0,T)}[\bs{\mu}^\epsilon(\chi_{\R^{3}\setminus\overline{\Omega}}+\chi_{\Omega^\epsilon})\nabla \widetilde{\varphi^\epsilon }+\chi_{\Omega^\epsilon}\M]\cdot\Bigl[\nabla_x \Psi^\epsilon+\frac{1}{\epsilon}\nabla_y \Psi^\epsilon\Bigr] \xdif x\xdif t=0.
	\end{eqnarray*}
	Using the extension operator $P^\epsilon$ and noting that for any function $f$ defined at $0$,
	$$\ds \chi_{\Omega^\epsilon}P^\epsilon f(v^\epsilon)=\chi_{\Omega^\epsilon}f(P^\epsilon v^{\epsilon})=\chi_{\Omega^\epsilon}f(\widetilde{v^{\epsilon}}),$$
	the latter equations are transformed into
	\begin{equation}\label{equ11.0}
		\begin{array}{ll}
			\ds\int_{\Omega_{T}}(\gamma\dep_t\widetilde{\M} +\theta_{c}(|\widetilde{\M}|^{2}-1)\widetilde{\M}+ \chi_{\Omega^\epsilon}G(\widetilde{v^{\epsilon}})\widetilde{\M}-\chi_{\Omega^\epsilon}\nabla \varphi^\epsilon)\cdot\p\left(t,x,\frac{x}{\epsilon}\right) \xdif x\xdif t
			\\
			\qquad \quad  +\ds\int_{\Omega_{T}}\chi_{\Omega^\epsilon}\A^{\epsilon}\nabla P^\epsilon\M\cdot\nabla_{x} \p\left(t,x,\frac{x}{\epsilon}\right)\ \xdif x\xdif t
			\\
			\qquad \quad  +\ds\frac{1}{\epsilon}\int_{\Omega_{T}}\chi_{\Omega^\epsilon}\A^{\epsilon}\nabla P^\epsilon\M\cdot\nabla_{y} \p\left(t,x,\frac{x}{\epsilon}\right) \xdif x\xdif t=0,
		\end{array}
	\end{equation}
	\begin{equation}\label{equ10}
		\begin{array}{ll}
			\ds\int_{\Omega_{T}}(\dep_{t}\widetilde{v^{\epsilon}}-
			\widetilde{\M}\cdot\dep_{t}\widetilde{\M})r\Bigl(t,x,\frac{x}{\epsilon}\Bigr) \xdif x\xdif t
			\\
			\qquad \quad 
			+\ds\int_{\Omega_{T}}\chi_{\Omega^\epsilon}\K^{\epsilon}g(G(\widetilde{v^{\epsilon}}))\nabla P^\epsilon v^{\epsilon}\cdot\nabla_{x} r\Bigl(t,x,\frac{x}{\epsilon}\Bigr) \xdif x\xdif t\\
			\qquad \quad  +\ds\frac{1}{\epsilon}\int_{\Omega_{T}}\chi_{\Omega^\epsilon}\K^{\epsilon}g(G(\widetilde{v^{\epsilon}}))\nabla P^\epsilon v^{\epsilon}\cdot\nabla_{y} r\Bigl(t,x,\frac{x}{\epsilon}\Bigr) \xdif x\xdif t=0 ,
		\end{array}
	\end{equation}
	%
	%
	\begin{equation}\label{equ20}
		\begin{array}{ll}
			\ds\int_{\R^3\times(0,T)}[\chi_{\R^{3}\setminus\overline{\Omega}}\bs{\mu}^\epsilon\nabla \widetilde{\varphi^\epsilon }+\chi_{\Omega}(\chi_{\Omega^\epsilon}\bs{\mu}^\epsilon\nabla\widetilde{\varphi^\epsilon }+\widetilde{\M})]
			\\
			\qquad \quad \ds \cdot\Bigl[\nabla_x \Psi\Bigl(t,x,\frac{x}{\epsilon}\Bigr)+\frac{1}{\epsilon}\nabla_y \Psi\Bigl(t,x,\frac{x}{\epsilon}\Bigr)\Bigr] \xdif x\xdif t=0.
		\end{array}
	\end{equation}
	The demonstration now consists in three main steps:
	
	\medskip
	
	\noindent\textbf{Step 1.} The first stage is devoted to  the convergence analysis of the nonlinear terms, based on the transformation $S^\epsilon$.
	
	\noindent\textbf{Limit of the term $\theta_{c}(|\widetilde{\M}|^{2}-1)\widetilde{\M}$:} Our estimates are not sufficient for relying on Lebesgue's dominated convergence theorem, as Berti did in \cite{berti}. 
	We thus pave the way for using Vitali's convergence theorem.
	Using the transformation $S^\epsilon$ and the properties \eqref{tra} and \eqref{47}, we obtain:
	\begin{eqnarray}
		&&\!\!\!  \!\!\!  \int_{\Omega_{T}}\theta_{c}(|\widetilde{\M}|^{2}-1)\widetilde{\M} \cdot\p\Bigl(t,x,\frac{x}{\epsilon}\Bigr) 			\nonumber \\
		&&  \!\!\! \!\!\! 
		=\theta_{c}\int_{\Omega_{T}}\int_{Y}(|\widetilde{\M}\circ S^\epsilon|^{2}\widetilde{\M}\circ S^\epsilon\cdot\p(S^\epsilon,y) -\theta_{c}\int_{\Omega_{T}}\int_{Y}\widetilde{\M}\circ S^\epsilon \cdot\p(S^\epsilon,y)    . 
		\label{48}
	\end{eqnarray}
	Convergences $\eqref{CON}_2$ and \eqref{conu} lead straightforward to:
	\begin{equation}\label{CVD1}
		\ds \lim_{\epsilon\to 0}\theta_{c}\int_{\Omega_{T}}\int_{Y}\widetilde{\M}\circ S^\epsilon \cdot\p\left(S^\epsilon,y\right) =\theta_{c}\int_{\Omega_{T}}\int_{Y}\chi_{Y^\ast} (y)\m(t,x) \cdot\p\left(t,x,y\right)  .
	\end{equation}
	Additional arguments are required to calculate the asymptotic behavior of the second term in the right-hand side of \eqref{48}.
	First, the strong convergence result stated in Lemma \ref{strong1} implies that, \textcolor{blue}{ up to a subsequence,} $\widetilde{\M}\circ S^\epsilon$ also converges almost everywhere in $\Omega_T \times Y$ to $\chi_{Y^\ast}\m$.
	Then, \textcolor{blue}{up to the same subsequence,} $|\widetilde{\M}\circ S^\epsilon|^{2}\widetilde{\M}\circ S^\epsilon$ converges almost everywhere in $\Omega_T \times Y$ to $\chi_{Y^\ast}\vert \m\vert^2\m$, hence also converges in measure. 
	Next, $|\widetilde{\M}\circ S^\epsilon|^{2}\widetilde{\M}\circ S^\epsilon$ is uniformly bounded in $L^q(\Omega_T \times Y)$ for some $q>1$ (see \textcolor{blue}{\eqref{L4} in} Lemma \ref{lem11}: $\widetilde{\M} \circ S^\epsilon$ is bounded in $L^4(\Omega_T \times Y)$).  It follows that the sequence of functions $(|\widetilde{\M}\circ S^\epsilon|^{2}\widetilde{\M}\circ S^\epsilon)$ has uniformly absolutely continuous integrals. 
	Vitali's convergence theorem applies: 
	$$ |\widetilde{\M}\circ S^\epsilon|^{2}\widetilde{\M}\circ S^\epsilon \to \chi_{Y^\ast}\vert \m\vert^2\m \mbox{ in } L^1(\Omega_T \times Y).$$
	The latter result is sufficient for proving that
	$$
	\lim_{\epsilon \to 0}  \int_{\Omega_{T}} \int_{Y} \bigl( |\widetilde{\M} \circ S^\epsilon|^{2} \, \widetilde{\M} \circ S^\epsilon \cdot \p(S^\epsilon, y) \bigr) = \int_{\Omega_{T}} \int_{Y} \chi_{Y^\ast} \vert \m \vert^{2}  \m \cdot \p(t,x, y) .
	$$
	Inserting the latter result in \eqref{CVD1} gives
	\begin{equation}
		\lim_{\epsilon \to 0} \int_{\Omega_{T}} \theta_{c} \bigl( |\widetilde{\M}|^{2} - 1 \bigr) \, \widetilde{\M} \cdot \p\Bigl( t,x, \frac{x}{\epsilon} \Bigr) = \int_{\Omega_{T}} \int_{Y} \chi_{Y^\ast} \theta_{c} \left( \vert \m |^{2} - 1 \right) \, \m \cdot \p(t,x,  y) .
	\end{equation}

	\noindent\textbf{Convergence of the term $\chi_{\Omega^\epsilon}G(\widetilde{v^{\epsilon}})\widetilde{\M}$:} 
	The  computation relies on Proposition \ref{prop} and Corollary \ref{cor}. 
	Indeed, a direct computation gives $0< G'\le 1/c_2$ and the Lipschitz property for $G$ follows from  the Mean Value Theorem.
	Therefore, according to Corollary \ref{cor}, the following convergence hold:
	\begin{equation}\label{conv4}
		\chi_{\Omega^\epsilon}G(\widetilde{v^\epsilon}) \overset{2}{\to} \chi_{Y^\ast}G( v)\text{ strongly two-scale in} \ L^2(\Omega_{T}\times Y).
	\end{equation}
	Bearing in mind that $\widetilde{\M}$ also two-scale converges  to $\chi_{Y^\ast} \m$ in $L^2(\Omega_{T} \times Y)$, one infers from  \eqref{conv4} (using $\ds(\chi_{Y^\ast})^2=\chi_{Y^\ast}$) that:
	\begin{equation}
		\ds\lim_{\epsilon\to 0}\int_{\Omega_{T}} \chi_{\Omega^\epsilon}G(\widetilde{v^{\epsilon}})\widetilde{\M}\cdot\p\Bigl(x,t,\frac{x}{\epsilon}\Bigr)=\int_{\Omega_{T}}\int_{Y} \chi_{Y^\ast}G(v)\m\cdot\p\left(t,x,y\right) .
	\end{equation}

	\noindent\textbf{Convergence of the term $\widetilde{\M}\cdot\dep_{t}\widetilde{\M}$:} The passage to the limit in this term is obtained directly by the convergences $\eqref{CON}_2$ and $\eqref{CON}_3$, along with the transformation $S^\epsilon$. Therefore, we obtain:
	\begin{equation}\label{dept}
		\lim_{\epsilon\to 0}\int_{\Omega_{T}} (\widetilde{\M}\cdot\dep_{t}\widetilde{\M}) r(x,t,y) =\int_{\Omega_{T}}\int_{Y} \chi_{Y^\ast}(\m\cdot\dep_{t}\m) r(x,t,y) .
	\end{equation}

	\noindent\textbf{Step 2.} 
	This is the classical step of any homogenization process where corrector results and auxiliary problems are derived.
	We begin with the two equations \eqref{equ11.0} and \eqref{equ20} multiplied by $\epsilon$ for capturing the microscale oscillations. 
	By  using results of Step 1, particularly to pass to the limit in equation \eqref{equ11.0},  applying the transformation $S^\epsilon$, along with \eqref{47} and the convergences $\eqref{correc}_2$ and $\eqref{correc}_3$ and the convergence $\eqref{CON}_2$, one obtains:  
	for a.e. $(t,x) \in (0,T)\times\mathbb{R}^3$, for any $\bar{\p}$ and $\bar{\Psi}$ in $C^{\infty}_\sharp(Y)$, 
	\begin{eqnarray}\label{E1}
		&& \int_{Y^{\ast}}\A(\nabla_{x} \m+\nabla_{y} \m_1)\cdot\nabla_{y} \bar{\p}(y) \xdif y=0,
		\\
		\label{E2}
		&& \int_{Y}[ \bs{\mu}(\chi_{\R^{3}\setminus\overline{\Omega}}+\chi_{\Omega}\chi_{Y^\ast})(\nabla_x \varphi+\nabla_y \varphi_1)+\chi_{\Omega}\chi_{Y^\ast}\m]\cdot\nabla_y \bar{\Psi}(y) \xdif y=0.
	\end{eqnarray}
	Equation \eqref{E1} being the variational formulation of the following problem:
	\begin{equation*}
		\left\{
		\begin{array}{ll}
			\ds -\dvg_y(\A(\nabla_{x} \m+\nabla_{y} \m_1))=0 &\mbox{in } Y^{\ast},\\
			\ds\A(\nabla_{x} \m+\nabla_{y} \m_1)\cdot\nu=0& \mbox{on }\dep \mathcal{H},
		\end{array}
		\right.
	\end{equation*}
	\textcolor{blue}{where function $\m_1$ is $Y$-periodic and may be chosen such that $\ds\mathcal{M}_{Y}(\m_1)=0$.} It can be expressed in the  form
	\begin{equation}\label{m1}
		\ds \m_1(t,x,y)=\sum_{i=1}^{3}\omega_i(y)\dep_{x_i}\m(t,x),
	\end{equation}
	where $\omega_i \in H^1_\sharp(Y)$ are the solution of the following cell problem
	\begin{equation}\label{echm}
		\left\{
		\begin{array}{ll}
			\ds -\dvg_y(\A\nabla_{y}( \omega_i+y_i))=0 &\mbox{in } Y^{\ast},\\
			\ds\A\nabla_{y}( \omega_i+y_i)\cdot\nu=0& \mbox{on }\dep \mathcal{H},\\
			\ds\mathcal{M}_{Y}(\omega_i)=0, \ \omega_i \mbox{ is $Y$-periodic}.
		\end{array}
		\right.
	\end{equation}
	Next, equation \eqref{E2} is the variational formulation of the following problem:
	\begin{equation*}
		\left\{
		\begin{array}{ll}
			\ds -\dvg_y(\bs{\mu}(\chi_{\R^{3}\setminus\overline{\Omega}}+\chi_{\Omega}\chi_{Y^\ast})(\nabla_x \varphi+\nabla_y \varphi_1)+\chi_{\Omega}\chi_{Y^\ast}\m))=0 &\mbox{in } \R^3\times Y,\\
			\ds\chi_{\Omega}(\bs{\mu}(\nabla_x \varphi+\nabla_y \varphi_1) +\m)\cdot\nu=0& \mbox{on }\R^3\times\dep \mathcal{H},
		\end{array}
		\right.
	\end{equation*}
	where  $\varphi_1$ is $Y$-periodic. Consequently,  $\varphi_1$ can be expressed in the form
	\begin{equation}\label{phi}
		\ds\varphi_1(t,x,y)=\sum_{i=1}^{3} \bar{\omega}^{1}_i(y)\dep_{x_i} \varphi(t,x) +\bar{\omega}^{2}_i(\m(t,x)\cdot\e_i),
	\end{equation}
	where $\bar{\omega}^{k}_i \in H^{1}_\sharp(Y),\ k=1,2,\ i= 1,2,3$, are the solutions of the cell problems:
	\begin{equation}\label{omega1}
		\left\{
		\begin{array}{ll}
			\ds -\dvg_y(\bs{\mu}(\chi_{\R^{3}\setminus\overline{\Omega}}+\chi_{\Omega}\chi_{Y^\ast})\nabla_{y}( \bar{\omega}^{1}_i+y_i))=0 &\mbox{in } \R^3\times Y,\\
			\ds\chi_{\Omega}\bs{\mu}\nabla_{y}(\bar{\omega}^{1}_i+y_i)\cdot\nu=0& \mbox{on }\R^3\times \dep \mathcal{H},\\
			\ds\mathcal{M}_{Y}(\omega^{1}_i)=0
			, \ \omega^{1}_i \mbox{ is $Y$-periodic}.
		\end{array}
		\right.
	\end{equation}
	and 
	\begin{equation}\label{omega2}
		\left\{
		\begin{array}{ll}
			\ds -\dvg_y(\bs{\mu}\nabla_{y}( \bar{\omega}^{2}_i+y_i))=0 &\mbox{in }  Y^{\ast},\\
			\ds \bs{\mu}\nabla_{y}(\bar{\omega}^{2}_i+y_i)\cdot\nu=0& \mbox{on }\dep \mathcal{H},\\
			\ds\mathcal{M}_{Y}(\bar{\omega}^{2}_i)=0, \ \bar{\omega}^{2}_i \mbox{ is $Y$-periodic}.
		\end{array}
		\right.
	\end{equation}

	We conclude this step with the expression of the corrector $v_1$. 
	However, before passing to the limit in equation \eqref{equ10}, one needs some compactness result for handling the nonlinear term $ \chi_{\Omega^\epsilon} g(G(\widetilde{v^{\epsilon}})) \nabla P^\epsilon v^{\epsilon}$. 
	Using  the Mean Value Theorem, one checks that $g \circ G$ is Lipschitz.
	Hence Corollary \ref{cor} applies:
	$$	\chi_{\Omega^\epsilon}g\circ G(\widetilde{v^\epsilon}) \overset{2}{\to} \chi_{Y^\ast}g\circ G( v)\text{ strongly two-scale in} \ L^2(\Omega_{T}\times Y),$$
	which is equivalent to
	$$ \chi(y) g\circ G(\widetilde{v^\epsilon}\circ S_\epsilon) \to \chi_{Y^\ast}g\circ G( v) \ \text{  in} \ L^2(\Omega_{T}\times Y).$$
	Now, multiplying equation \eqref{equ20} by $\epsilon$,  using the previous result along with the convergences $\eqref{strong1}_1$, $\eqref{correc}_1$, and \eqref{dept}, one obtains:
	\begin{equation*}
		\ds\int_{\Omega_{T}}\int_{Y}\chi_{Y^\ast}\K g(G(v))(\nabla_x v+\nabla_y v_1)\cdot\nabla_{y} r(t,x,y) \xdif x\xdif y \xdif t=0.
	\end{equation*}
	Notice that $G\circ g(v) >0$ and does not depend on $y$. Similar computations \textcolor{blue}{as} before lead to express $v_1$ as
	\begin{equation}\label{v11}
		\ds v_1(t,x,y)=\sum_{i=1}^{3}\widehat{\omega}_i(y)\dep_{x_i} v(t,x),
	\end{equation}
	where $\widehat{\omega}_i\in H^1_\sharp(Y)$, $i=1,2,3,$ is the solution of the following cell problem: 
	\begin{equation}\label{omegav}
		\left\{
		\begin{array}{ll}
			\ds -\dvg_y(\K\nabla_{y}(\widehat{\omega}_i+y_i))=0 &\mbox{in } Y^{\ast},\\
			\ds \K\nabla_{y}(\widehat{\omega}_i+y_i)\cdot\nu=0& \mbox{on }\dep \mathcal{H},\\
			\ds\mathcal{M}_{Y}(\widehat{\omega}_i)=0, \ \widehat{\omega}_i \mbox{ is $Y$-periodic}.
		\end{array}
		\right.
	\end{equation}

	\noindent\textbf{Step 3.} Finally, we gather all the previous results for passing to the limit $\epsilon \to 0$  in equations \eqref{equ11.0}-\eqref{equ20}. 
	We pick test functions $\p^{\epsilon}=\p$, $r^{\epsilon}=r$ in $C^{\infty}(\overline{\Omega_{T}})$, and $\Psi^{\epsilon}=\Psi$ in $\mathcal{D}((0,T)\times\R^3)$. 
	We obtain at the limit
	\begin{eqnarray*}
		&& \int_{\Omega_{T}}(\gamma\dep_t\m +\theta_{c}(\m|^{2}-1)\m+ G(v)\m-\mathcal{M}_{Y^\ast}(\nabla_x \varphi+\nabla_x  \varphi_1) )\cdot\p\, \xdif x\xdif t
		\\
		&& \qquad		\qquad
		+\int_{\Omega_{T}}(\mathcal{M}_{Y^\ast}(\A(x,y)(\nabla_x \m+\nabla_y \m_1))\cdot\nabla \p\ \xdif x\xdif t=0,
		\\
		&& \int_{\Omega_{T}}(\dep_{t}v-
		\m\cdot\dep_{t}\m)r\, \xdif x\xdif t+\ds\int_{\Omega_{T}}\mathcal{M}_{Y^\ast}(\K(x,y) G(v)(\nabla v+\nabla_y v_1)) \cdot\nabla r\xdif x\xdif t=0,
		\\
		&& \int_{\R^3\times(0,T)}\Bigr[\int_{Y}(\chi_{\R^{3}\setminus\overline{\Omega}}+\chi_{\Omega}\chi^\ast)\bs{\mu}(x,y)(\nabla_x\varphi+\nabla_y\varphi_1)\xdif y+\chi_{\Omega}\bar{\chi}\m\Bigr]\cdot\nabla \Psi \xdif x\xdif t=0.
	\end{eqnarray*}
	Using the expressions \eqref{m1}, \eqref{phi} and \eqref{v11}, bearing in mind \eqref{convinit} for  the initial conditions, one recovers the homogenized system presented in \eqref{modelfit100}.
	According to \eqref{convinit}, the limit of the initial conditions are such that
	\begin{eqnarray*}
		&& \chi^\epsilon \widetilde \m_0^\epsilon \rightharpoonup \overline \chi \m_0 \mbox{ and }
		\chi^\epsilon \widetilde v_0^\epsilon \rightharpoonup \int_Y \chi^\ast v_0 \, dy \mbox{ weakly in } L^2(\Omega_T). 
	\end{eqnarray*}

	One turns back to the result announced in Theorem \ref{main}, by setting $\theta = G(v)$.
\end{proof}

\section{Conclusion}\label{conclusion}
	The Curie temperature of ferromagnetic materials is known to depend on several factors, including the precise material composition, the presence of impurities and the manufacturing process.  Porosity, including perforations, is another such factor.  In this article, we derive an effective model that enables the accurate calculation of the parameters governing the ferromagnetic-paramagnetic transition, explicitly accounting for the microscopic geometry and the magnetic permeability tensor. 
It is worth noticing that this explicit calculation, which can be done with standard numerical tools, is valid for any type of material and can avoid or be compared to the vast literature giving experimental results on the Curie temperature in nano-porous materials (we only quote \cite{shi} which is focused on the influence of geometric characteristics).
This is all the more important as we have demonstrated in this article that, in all generality, the concept of Curie temperature in a perforated material has to be replaced by a Curie temperature tensor which  essentially averages the coupling of macro- and micro-scales (even if we assume a micro-model with a scalar Curie temperature contrary to a more proper definition as in \cite{Cza}). While the assumption of periodicity might appear restrictive, studies of homogenization in other areas ({\it e.g.}, porous media flows) have demonstrated that our results can be extended to more general configurations, such as stochastic distributions (see \cite{Bourgeat}). Furthermore, the accuracy of the effective model can be assessed as a function of the parameter $\epsilon$, which characterizes the inclusion size (see, {\it e.g.}, work on correctors).  Future work will address similar analyses with other inclusion morphologies.\vskip6pt

\enlargethispage{20pt}



\vskip2pc


%
%
%

\end{document}